\documentclass[12pt]{amsart}
\usepackage{amssymb,amsmath,amsthm,bm}
\usepackage[colorinlistoftodos,prependcaption,textsize=tiny]{todonotes}
 \definecolor{darkblue}{RGB}{0,0,160}

\usepackage{txfonts} 
\usepackage[colorlinks=true,allcolors=darkblue]{hyperref}

\usepackage[T1]{fontenc}

\usepackage{color}

\def\cic{\mathbf}
\def\eps{\varepsilon}

\def\d{{\rm d}}

\def\R {\mathbb{R}}

\def \l {\langle}
\def \r {\rangle}

\def \and{\qquad\text{and}\qquad}

\newcommand{\supp}{\mathrm{supp}\,}

\newcounter{thms}

\newcounter{other}
\numberwithin{other}{section}
\newtheorem{proposition}[subsection]{Proposition}
\newtheorem{theorem}[thms]{Theorem}
\newtheorem*{theorem*}{Theorem}
\newtheorem*{proposition*}{Proposition}
\newtheorem{corollary}{Corollary}
\numberwithin{corollary}{thms}
\newtheorem{lemma}[subsection]{Lemma}
\theoremstyle{definition}
\newtheorem{remark}[subsection]{Remark}

\def\vint_#1{\mathchoice%
      {\mathop{\kern 0.2em\vrule width 0.6em height 0.69678ex depth -0.58065ex
              \kern -0.8em \intop}\nolimits_{\kern -0.4em#1}}%
      {\mathop{\kern 0.1em\vrule width 0.5em height 0.69678ex depth -0.60387ex
              \kern -0.6em \intop}\nolimits_{#1}}%
      {\mathop{\kern 0.1em\vrule width 0.5em height 0.69678ex depth -0.60387ex
              \kern -0.6em \intop}\nolimits_{#1}}%
      {\mathop{\kern 0.1em\vrule width 0.5em height 0.69678ex depth -0.60387ex
              \kern -0.6em \intop}\nolimits_{#1}}}
\def\vintslides_#1{\mathchoice%
      {\mathop{\kern 0.1em\vrule width 0.5em height 0.697ex depth -0.581ex
              \kern -0.6em \intop}\nolimits_{\kern -0.4em#1}}%
      {\mathop{\kern 0.1em\vrule width 0.3em height 0.697ex depth -0.604ex
              \kern -0.4em \intop}\nolimits_{#1}}%
      {\mathop{\kern 0.1em\vrule width 0.3em height 0.697ex depth -0.604ex
              \kern -0.4em \intop}\nolimits_{#1}}%
      {\mathop{\kern 0.1em\vrule width 0.3em height 0.697ex depth -0.604ex
              \kern -0.4em \intop}\nolimits_{#1}}}

\newcommand{\aveint}[2]{\mathchoice%
      {\mathop{\kern 0.2em\vrule width 0.6em height 0.69678ex depth -0.58065ex
              \kern -0.8em \intop}\nolimits_{\kern -0.45em#1}^{#2}}%
      {\mathop{\kern 0.1em\vrule width 0.5em height 0.69678ex depth -0.60387ex
              \kern -0.6em \intop}\nolimits_{#1}^{#2}}%
      {\mathop{\kern 0.1em\vrule width 0.5em height 0.69678ex depth -0.60387ex
              \kern -0.6em \intop}\nolimits_{#1}^{#2}}%
      {\mathop{\kern 0.1em\vrule width 0.5em height 0.69678ex depth -0.60387ex
              \kern -0.6em \intop}\nolimits_{#1}^{#2}}}

\numberwithin{equation}{section}

\title[Sparse bounds  via Fourier transform  ]{Sparse bounds for maximal rough singular integrals\\ via the Fourier transform}

 \author[F.\ Di Plinio]{Francesco Di Plinio} \address{\noindent (FDP) Department of Mathematics, University of Virginia,      Box 400137, Charlottesville, VA 22904, USA   }
 \email{francesco.diplinio@virginia.edu}
\author[T.P.\ Hyt\"onen]{Tuomas P.\ Hyt\"onen}\address{
 \noindent
(TPH) Department of Mathematics and Statistics, P.O.B. 68, FI-00014 University of Helsinki, Finland}
\email{tuomas.hytonen@helsinki.fi}
  \author[K.\ Li]{Kangwei Li} \address{\noindent (KL) BCAM, Basque Center for Applied Mathematics, Bilbao, Spain}
 \email{kli@bcamath.org}

\thanks{F. Di Plinio was partially
supported by the National Science Foundation under the grants NSF-DMS-1500449 and  NSF-DMS-1650810.
 T. Hyt\"onen was supported by the Finnish Centre of Excellence in Analysis and Dynamics Research. 
    K. Li was supported by Juan de la Cierva-Formaci\'on 2015 FJCI-2015-24547.
     Some of the results were discovered at the Mathematical Sciences Research Institute (MSRI) in Berkeley,  during the workshop ``Recent trends in harmonic analysis'', in which Di Plinio and Hyt\"onen participated with the support of the MSRI and the Clay Mathematics Institute.   Di Plinio and Li  are also supported by the Severo Ochoa Program SEV-2013-0323 and by Basque Government BERC Program 2014-2017.}

\begin{document}
  \begin{abstract}
We prove that the class of convolution-type   kernels satisfying suitable decay conditions of the Fourier transform, appearing in the works of Christ \cite{Ch88}, Christ-Rubio de Francia \cite{CRub} and Duoandi\-koet\-xea-Rubio de Francia \cite{DR} gives rise to maximally truncated singular integrals satisfying a sparse bound by $(1+\eps,1+\eps)$-averages for all $\eps>0$, with linear growth in $\eps^{-1}$. This is an extension of the sparse domination principle by   Conde-Alonso, Culiuc, Ou and the first author \cite{CoCuDPOu}  to  maximally truncated singular integrals. 
Our results  cover the rough homogeneous singular integrals on $\R^d$
\[
T_\Omega f(x)=  \mathrm{p.v.}\int_{\R^d} f(x-t) \frac{\Omega(t/|t|)}{|t|^d} \, \d t
\] with  angular part $\Omega\in L^\infty(S^{d-1})$ and having vanishing integral on the sphere.
Consequences of our sparse bound include novel quantitative weighted norm estimates as well as Fefferman-Stein type inequalities. In particular, we obtain that the $L^2(w)$ norm of the maximal truncation of $T_\Omega$ depends quadratically on the Muckenhoupt constant $[w]_{A_2}$, extending a result originally by Hyt\"onen, Roncal and Tapiola \cite{HRT}.
 A suitable convex-body valued version of the sparse bound is also deduced and employed towards novel matrix weighted norm inequalities for 
the maximal truncated rough homogeneous  singular integrals.
  Our result is quantitative, but even the qualitative statement is new, and the present approach via sparse domination is the only one currently known for the matrix weighted bounds of this class of operators. 

   \end{abstract}

\maketitle

\section{Introduction and main results}
 Let $\eta\in (0,1)$. A countable collection $\mathcal S$ of cubes of $\R^d$ is said to be $\eta$-sparse if there exist    measurable sets $\{E_I:I \in \mathcal S\}$ such that
\[
E_I \subset I, \, |E_I| \geq \eta |I|, \qquad I,J\in \mathcal S, I \neq J \implies E_I \cap E_J  = \varnothing.
\]
 Let $T$ be a sublinear operator mapping the space $L^\infty_0(\R^d)$ of complex-valued, bounded and compactly supported functions on $\R^d$ into locally integrable functions. We say that $ T$ has the sparse $(p_1,p_2)$ bound   \cite{CKL} if there exists a constant $C>0$ such that for all $f_1,f_2\in L^\infty_0(\R^d)$ we may find a $\frac12$-sparse collection $\mathcal S=\mathcal S(f_1,f_2)$ such that
\[
\begin{split}
\left|\l Tf_1 , f_2\right\r| \leq C  \sum_{Q\in \mathcal S} |Q| \prod_{j=1}^2 \l f_j \r_{p_j,Q}\end{split}
\]
in which case we denote by $\|T  \|_{(p_1,p_2), \mathsf{sparse}}$ the least such constant $C$. As customary,
\[\l f \r_{p,Q} = \frac{\|f\cic{1}_Q\|_p}{|Q|^{\frac1p}}, \qquad p \in (0,\infty].
\]

  Estimating the sparse norm(s) of   a sublinear or multisublinear operator entails a sharp control over the behavior of such operator in weighted $L^p$-spaces; this theme has been recently pursued by several authors, see for instance \cite{BFP,CuDPOu,LSMF,LMena,LS,ZK2}. This sharp control is exemplified in the  following proposition, which is a collection of known facts from the indicated references.
\begin{proposition}\label{general}
Let $T$ be a sublinear operator on $\R^d$  mapping    $L^\infty_0(\R^d)$ to  $L^1_{\mathrm{loc}}(\R^d)$. Then the following hold.
\begin{itemize}
\item[1.] \cite[Appendix B]{CoCuDPOu} Let $1\leq p_1, p_2<\infty$. There is an absolute  constant $C_{p_2}>0$ such that \[
\|T:L^{p_1}(\R^d)\to L^{p_1,\infty}(\R^d)\|\leq C_{p_2} \|  T \|_{(p_1,p_2), \mathsf{sparse}} \]
\item[2.] \cite[Proposition 4.1]{DPLer2013} If
\begin{equation}
\label{Psi}
 \Psi(t):= \| T \|_{\left(1+\frac{1}{t} ,  1+\frac{1}{t}\right),\mathsf{sparse}}  <\infty \qquad \forall t>1,  \end{equation} then there is an absolute  constant $C>0$ such that
\[
\|T\|_{ L^{2}(w,\R^d) } \leq C [w]_{A_2}   \Psi\left( C [w]_{A_2}   \right). \qquad
\]
In particular,
\[
\sup_{t>1} \Psi(t) <\infty \implies
 \|T\|_{ L^{2}(w,\R^d)} \leq C [w]_{A_2}
     .\]
\end{itemize}
\end{proposition}
In this article, we are concerned with the sparse norms \eqref{Psi} of a class  of  convolution-type singular integrals   whose systematic study dates back to   the celebrated works by  Christ \cite{Ch88}, Christ-Rubio de Francia \cite{CRub}, and Duoandikoetxea-Rubio de Francia \cite{DR}, admitting a decomposition with good decay properties of the Fourier transform. To wit, let $\{K_s:\R^d\to \mathbb C, s\in \mathbb Z\}$ be a sequence of (smooth) functions with the   properties that
\begin{equation}
\label{assker}
\begin{split}
&\mathrm{supp}\, K_s\subset A_s:=\{x\in \R^d: 2^{s-4}< |x|_\infty< 2^{s-2}\},
\\
&\sup_{s\in \mathbb Z}  2^{sd}\|K_s\|_\infty     \leq 1, \\
& \sup_{s\in \mathbb Z} \sup_{\xi \in \R^d} \max\big\{|2^s \xi|^\alpha,|2^s \xi|^{-\alpha}\big\} |\widehat{K_s}(\xi)| \leq 1,
\end{split}
\end{equation}
for some $\alpha>0$.
We  consider truncated singular integrals of the type
\[
T f(x,t_1,t_2) = \sum_{t_1 < s\leq t_2 } K_s*f(x), \quad t_1,t_2 \in \mathbb Z,\]
and their  maximal version
\begin{equation}\label{deftstar}
T_\star f(x):=\sup_{t_1\leq t_2 } \left|  T  f(x,t_1,t_2) \right|.
\end{equation}
 \begin{theorem} \label{theoremRK}
   For all   $0<\eps<1$   \[
 \|  T_\star  \|_{\left(1+\eps , 1+\eps\right),\mathsf{sparse}}  \lesssim  \frac{1}{\eps},\]
with absolute dimensional implicit constant, in particular uniform over families $\{K_s\}$ satisfying \eqref{assker}.
\end{theorem}
Theorem \ref{theoremRK} entails immediately a variety of novel corollaries involving weighted norm inequalities for the maximally truncated operators $T_\star$. In addition to, for instance, those obtained by suitably applying the points of Proposition \ref{general},
 we also   detail the   quantitative estimates below, whose proof will be given in Section \ref{sec:quantitative}.
\begin{theorem}\label{cor}
Let $T$ be a sublinear operator satisfying the sparse bound \eqref{Psi}
with $\Psi (t)\leq Ct$.
\begin{itemize}
\item[1.] For any $1<p<\infty$,
\[
\|T \|_{L^p(w)}\lesssim   [w]_{A_p}^{\frac 1p}([w]_{A_p}^{\frac 1{p'}}+[\sigma]_{A_\infty}^{\frac 1p}) \max\{[\sigma]_{A_\infty}, [w]_{A_\infty}\}
\]with implicit constant possibly depending on $p$ and dimension $d$;
in particular,
\begin{equation}
\label{quadest}
\|T \|_{L^p(w)}\lesssim [w]_{A_p}^{2\max\{1,\frac 1{p-1}\}}.
\end{equation}
\item[2.] The  Fefferman-Stein type inequality
\[
\|T f\|_{L^p(w)}\lesssim p^2 (p')^{\frac 1{p}} (r')^{1+\frac 1{p'}}\|f\|_{L^p(M_r w)}, \,\, r<p
\]holds with implicit constant possibly depending on      $d$ only.
\item[3.] The   $A_q$-$A_\infty$ estimate\begin{equation}
\label{estuntcase}
\|T f\|_{L^p(w)}\lesssim [w]_{A_q}^{\frac 1p} [w]_{A_\infty}^{1+\frac 1{p'}}\|f\|_{L^p(w)},
\end{equation}
holds for $q<p$ and $w\in A_q$, with implicit constant possibly depending on $p,q$ and   $d$ only.
\item[4.] The following Coifman-Fefferman type inequality  \[
\|T f \|_{L^p(w)}\lesssim \frac{[w]_{A_\infty}^2}{\varepsilon}\|M_{1+\varepsilon} f\|_{L^p( w)}
\]
holds for all $\varepsilon>0$ with implicit constant possibly depending on $p$ and    $d$ only.
\end{itemize}
\end{theorem}
\begin{remark}
Take $\Omega:S^{d-1}\to \mathbb C$ with $\|\Omega\|_\infty\leq 1$ and  having vanishing integral on $S^{d-1}$,
and consider the associated truncated  integrals and their maximal function
\begin{equation}
\label{Tomegas}
T_{\Omega,\delta} f(x): =  \int_{\delta<|t|<\frac1\delta}  f(x-t) \frac{\Omega(t/|t|)}{|t|^d} \, \d t, \qquad
T_{\Omega,\star} f(x) : =\sup_{\delta>0}\left|  T_{\Omega,\delta} f(x)\right|, \qquad x\in \R^d.
\end{equation}
 It is well known-- for instance, see the recent contribution \cite[Section 3]{HRT}-- that
\[
T_{\Omega,\star} f(x)
 \lesssim  \mathrm{M} f(x) +  T_{\star} f(x), \qquad x \in \R^d\]
 with $T_\star $ being defined as in \eqref{deftstar} for a suitable choice of $\{K_s: s\in \mathbb Z\}$ satisfying \eqref{assker} with $\alpha= \frac1d$.  As $ \| \mathrm M  \|_{\left(1 , 1 \right),\mathsf{sparse}} \lesssim 1$, a corollary of Theorem \ref{theoremRK} is that
  \begin{equation}
\label{roughOmega}
 \| T_{\Omega,\star}  \|_{\left(1+\eps , 1+\eps\right),\mathsf{sparse}}  \lesssim \frac{1}{\eps}\end{equation}
 as well. The main result of \cite{CoCuDPOu} is  the  stronger control
\begin{equation}
\label{nontrunc}
\sup_{\delta>0}  \| T_{\Omega,\delta}  \|_{\left(1 , 1+\eps\right),\mathsf{sparse} } \lesssim  \frac{1}{\eps}.
\end{equation}
The above estimate, in particular, is stronger that the uniform weak type $(1,1)$ for the operators $T_{\Omega,\delta}$, a result originally due to Seeger \cite{Seeger}. As the weak type $(1,1)$ of $T_{\Omega,\star}$ under no additional smoothness assumption on $\Omega$ is a difficult open question, estimating the $(1,1+\eps)$ sparse norm of $ T_{\Omega,\star}$ as in \eqref{nontrunc} seems out of reach.

The study of sharp weighted norm inequalities for $T_{\Omega,\delta}$ (the uniformity in $\delta$ is of course relevant here) was initiated in  the recent article \cite{HRT} by Hyt\"onen, Roncal and Tapiola. Improved quantifications    have been obtained  in \cite{CoCuDPOu} as a consequence of the domination result \eqref{nontrunc}, and further weighted estimates-- including a Coifman-Fefferman type inequality, that is a norm control of $T_{\Omega,\delta}$ by $\mathrm{M}$ on all $L^p(w)$, $0<p<\infty$ when $w\in A_\infty$-- have been later derived from \eqref{nontrunc} in the recent preprint by the third named author, P\'erez, Roncal and Rivera-Rios \cite{LPRR}.

Although \eqref{roughOmega} is a bit weaker than \eqref{nontrunc}, we see from comparison of \eqref{quadest} from Theorem \ref{cor} with the results of \cite{CoCuDPOu,LPRR} that the quantification of the $L^2(w)$-norm dependence on  $[w]_{A_2}$ entailed by the two estimates  is the same-- quadratic; on the contrary, for $p\neq 2$, \eqref{nontrunc} yields the  better estimate $\|T_{\Omega,\delta}\|_{L^p(w)}\lesssim [w]_{A_p}^{p'}.$ We also observe that the proof of the mixed estimate \eqref{estuntcase} actually yields the following estimate for the non-maximally truncated operators, improving the previous estimate given in \cite{LPRR}
\[
\|T_{\Omega,\delta}f\|_{L^p(w)}\lesssim   [w]_{A_q}^{\frac 1p} [w]_{A_\infty}^{\frac 1{p'}}\|f\|_{L^p(w)}.
\]
Finally, we emphasize that \eqref{roughOmega} also yields a precise dependence on $p$ of the unweighted $L^p$ operator norms. Namely, from the sparse domination, we get 
\begin{equation}\label{eq:unweightedp}
\|T_{\Omega,\star}\|_{L^p(\R^d)\to L^{p,\infty}(\R^d) }\lesssim   \max\{p,p'\}, \qquad 
\| T_{\Omega,\star}\|_{L^p(\R^d)}\lesssim pp'\max\{p,p'\}
\end{equation} with absolute dimensional implicit constant,
which improves on the implicit constants in \cite{DR}. Moreover, we note that the main result of  \cite{LuquePR} implies that if \eqref{eq:unweightedp} is sharp, then our quantitative weighted estimate \eqref{quadest} is also sharp. 
\end{remark}

\begin{remark}
Comparing Theorem \ref{theoremRK} with the sparse domination formula for commutators of Calder\'on-Zygmund operators with $BMO$ symbols \cite{LOR}, all our weighted corollaries hold for commutators as well, with the help of John-Nirenberg inequality.
\end{remark}

 \subsection{Matrix  weighted estimates for vector valued rough singular integrals} Let $(e_\ell)_{\ell=1}^n$,    $\l\cdot ,\cdot \r_{\mathbb F^n}$ and $|\cdot|_{\mathbb F^n}$ be the canonical basis, scalar product and   norm on $\mathbb{F}^n$ over $\mathbb F$, where    $\mathbb F\in \{\R,\mathbb C\}$.  
  A recent trend in Harmonic Analysis-- see, among others, \cite{BCTW,BPW,CuDPOu2,HPV,NPTV}-- is the study of quantitative  matrix weighted norm inequalities for the canonical extension of the (integral)   linear operator $T$
\[
\left\l T   f(x),e_\ell \right\r_{\mathbb F^n}:=
\left\l T\otimes {\mathrm{Id}}_{\mathbb F_n} f  (x),e_\ell \right\r_{\mathbb F^n} =  T \big(\l f, e_{\ell}\r_{\mathbb F^n}\big) (x), \qquad x\in \mathbb \R^d \] 
to $\mathbb F^n$-valued functions $f$.  In Section \ref{svv} of this paper, we introduce an  $L^p$, $p>1$,  version of the  convex body averages  first brought  into the sparse domination context by Nazarov, Petermichl, Treil and Volberg \cite{NPTV}, and use them to produce a vector valued version of Theorem \ref{theoremRK}. As a corollary, we obtain quantitative matrix weighted estimates for the maximal truncated vector valued extension  of  the rough singular integrals $T_{\Omega,\delta}$ from \eqref{Tomegas}. In fact, the next corollary is a special case of the more precise Theorem \ref{theoremW} from Section \ref{svv}.\setcounter{thms}{5}
\begin{corollary}
\label{corollaryW} Let  $W$  be a positive semidefinite and locally integrable $\mathcal L(\mathbb{F}^{n})$-valued function on $\R^d$ and  $T_{\Omega,\delta}$ be as in  \eqref{Tomegas}.   Then  
\begin{equation}
\label{wintroeq}\left\| \sup_{\delta>0} \big|  W^{\frac12}  T_{\Omega,\delta}   f  \big|_{\mathbb F^n}\right\|_{L^2(\R^d)} \lesssim  [W]_{A_2}^{\frac52}\left\|   \big|  W^{\frac12}      f \big|_{\mathbb F^n}\right\|_{L^2(\R^d)}\end{equation} with implicit  constant   depending on $d,n$ only,
where the \emph{matrix $A_2$ constant} is given by  \[
[W]_{A_2}:=\sup_{Q \textrm{\emph{ cube of }} \R^d}\left\| \left(\frac{1}{|Q|}\int_{Q} W(x) \, \d x\right)^{\frac12} \left(\frac{1}{|Q|}\int_{Q} W^{-1}(x) \, \d x\right)^{\frac12}\right\|_{\mathcal L(\mathbb{F}^n)}^2.
\]
\end{corollary}
As the left hand side of \eqref{wintroeq} dominates the matrix weighted norm of the vector valued maximal operator first studied by Christ and Goldberg in \cite{CG},
the finiteness of  $[W]_{A_2}$ is actually necessary for the estimate to hold.  
To the best of the authors' knowledge, Theorem \ref{theoremW}  has no precedessors, in the sense that no matrix weighted norm inequalities for vector rough singular integrals were known before, even in qualitative form. At this time we are unable to assess whether the power $\frac52$ appearing in \eqref{wintroeq} is optimal. For comparison, if the angular part $\Omega $ is H\"older continuous, the currently best known result \cite{NPTV} is that \eqref{wintroeq} holds with power $\frac32$; see also \cite{CuDPOu2}.
\setcounter{thms}{2}

\subsection{Strategy of proof of the main results}We will obtain  Theorem \ref{theoremRK} by an application of an abstract sparse domination principle, Theorem \ref{theoremABS} from Section \ref{S2}, which is a modification of \cite[Theorem C]{CoCuDPOu}. At the core of our approach lies a special configuration of stopping cubes, the so-called \emph{stopping collections} $\mathcal Q$, and their related atomic spaces. The necessary definitions, together with a useful interpolation principle for the atomic spaces, appear in Section \ref{sectionSC}.
\setcounter{other}{4}
 In essence, Theorem \ref{theoremABS}  can be summarized by the inequality
\[
  \| T_\star  \|_{\left(p_1,p_2\right),\mathsf{sparse}}
\lesssim
 \| T_\star  \|_{\mathcal L(L^{2}(\R^d)) } +   \sup_{\mathcal Q,t_1,t_2}\left( \|\mathcal Q _{ t_1}^{t_2}\|_{\dot{\mathcal X}_{p_1}  \times   \mathcal Y_{p_2} } + \|\mathcal Q _{ t_1}^{t_2}\|_{   \mathcal Y_{\infty} \times \dot{\mathcal X}_{p_2} }\right)\]
where the supremum is taken over all stopping collections $\mathcal Q$ and all   measurable linearizations  of the truncation parameters $t_1,t_2$, and $\mathcal Q _{ t_1}^{t_2}$ are suitably adapted localizations of (the adjoint form to the linearized versions of) $ T_\star$. In Section \ref{S3}, we prove the required uniform estimates for the localizations $\mathcal Q _{ t_1}^{t_2}$ coming from Dini-smooth kernels.
 The proof of Theorem \ref{theoremRK} is given in Section \ref{S4}, relying upon the estimates of Section \ref{S3} and the   Littlewood-Paley decomposition of the convolution kernels \eqref{assker} whose first appearance dates back to  \cite{DR}.

\begin{remark} We remark that while this article was being finalized, an alternative proof of  \eqref{nontrunc} was given by Lerner \cite{Ler2017}. It is of interest whether the strategy of \cite{Ler2017}, relying on bumped bilinear grand local maximal functions, can be applied towards estimate \eqref{roughOmega} as well.
\end{remark}

\subsection*{Notation}
 With $q'=\frac{q}{q-1}$ we indicate the Lebesgue dual exponent to $q\in (1,\infty)$, with the usual extension $1'=\infty$, $ \infty'=1$.
The center and the (dyadic) scale of  a cube $Q\in \R^d$ will be denoted by  $c_Q$   and  $s_Q $ respectively, so that $|Q|=2^{ s_Q d}$.
We use the notation
\[
\mathrm{M}_{p}(f)(x) =\sup_{Q \subset \R^d} \l f \r_{p,Q} \cic{1}_Q(x)
\] for the  $p$-Hardy Littlewood maximal function and write $\mathrm{M}$ in place of $\mathrm{M}_1$.
Unless otherwise specified, the almost inequality signs $\lesssim$  imply   absolute dimensional constants which may  be  different at each occurrence.\subsection*{Acknowledgments}
This work was completed during F.\ Di Plinio's stay at the Basque Center for Applied Mathematics (BCAM), Bilbao as a visiting fellow. The author gratefully acknowledges the kind hospitality of the staff and researchers at BCAM and in particular of Carlos P\'erez.
The authors also    want to thank Jos\'e Conde-Alonso, Amalia Culiuc,\ Yumeng Ou and Ioannis Parissis for several inspiring discussions on sparse domination principles.
\section{Stopping collections and interpolation in localized spaces} \label{sectionSC}
The notion of   stopping collection $\mathcal Q$ with top  the (dyadic) cube  $Q$ has been introduced in \cite[Section 2]{CoCuDPOu}, to which we send for details. Here, we recall that such a $\mathcal Q$ is a collection of pairwise disjoint   dyadic cubes contained in $3Q$ and satisfying suitable Whitney type properties. More precisely,
\begin{align}
\label{stop1}
& \bigcup_{L\in \mathsf{c} \mathcal Q} 9L   \subset
\mathsf{sh}\mathcal Q:=\bigcup_{L\in \mathcal Q} L \subset 3Q, \qquad \mathsf{c} \mathcal Q:= \{L\in \mathcal Q: 3L \cap 2Q\neq \emptyset \} ;
\\ \label{stop2}
&
  L,L'\in \mathcal Q, \, L \cap L'\neq \emptyset \implies L=L';
\\ \label{stop3}
&
L\in \mathcal Q, L' \in N(L)  \implies |s_L-s_{L'}|\leq 8, \qquad N(L):= \{ L'\in \mathcal Q: 3L\cap 3L' \neq \emptyset \}.
\end{align}
A consequence of \eqref{stop3} is that the cardinality of $N(L)$ is bounded by an absolute constant.

 The spaces $\mathcal Y_p(\mathcal Q), \mathcal X_p(\mathcal Q), \dot{\mathcal X}_p(\mathcal Q) $ have  also been defined in \cite[Section 2]{CoCuDPOu}: here we recall that $
\mathcal Y_p(\mathcal Q)$ is the subspace of $L^p(\R^d)$ of functions satisfying
\begin{equation}\label{Yq}
\begin{split}
& \supp f \subset 3Q, \\ & \infty>
\|f\|_{\mathcal Y_p(\mathcal Q) }:=\begin{cases} \displaystyle \max\left\{\left\|f\cic{1}_{\R^d\setminus  \mathsf{sh} \mathcal Q }\right\|_{\infty} ,\,\sup_{L\in \mathcal Q}\,\inf_{x\in \widehat L } \mathrm{M}_p f(x) \right\} & p<\infty \\ \|f \|_{\infty} & p=\infty
\end{cases} \end{split} \end{equation}
where   $\widehat L$ stands for the (non-dyadic) $2^{5}$-fold dilate of $L$,
and that  ${\mathcal X_p(\mathcal Q) }$ is the subspace of $\mathcal Y_p(\mathcal Q)$ of functions satisfying
\begin{equation}
\label{bell}
b=\sum_{L\in \mathcal Q} b_L, \qquad \supp  b_L \subset L.
\end{equation}
Finally, we write $b\in \dot{\mathcal X}_{p}(\mathcal Q)$ if $b\in  {\mathcal X}_{p}(\mathcal Q)$ and each $b_L$ has mean zero.  We will omit $(\mathcal Q)$ from the subscript of the norms whenever the stopping collection $\mathcal Q$ is clear from context.

There is a natural interpolation procedure involving the $\mathcal Y_p$-spaces. We do not strive for the most general result but restrict ourselves to proving a significant example, which is also of use to us in the proof of Theorem \ref{theoremRK}.
\begin{proposition} \label{interprop}
Let $B$ be a bisublinear form  and $A_1,A_2$ be positive constants such that the estimates
\[
\begin{split}
|B(b,f)| \leq A_1 \|b_1\|_{\dot{\mathcal X}_1(\mathcal Q) } \|f\|_{\mathcal Y_1(\mathcal Q) }, \qquad |B(g_1,g_2)| \leq A_2 \|g_1\|_{{\mathcal Y}_2(\mathcal Q) } \|g_2\|_{\mathcal Y_2(\mathcal Q) }
\end{split}
\]
Then for all $0<\eps< 1$
\[
\begin{split}
|B(f_1,f_2)| \lesssim (A_1)^{1-\eps} (A_2)^{\eps}  \|f_1\|_{\dot{\mathcal X}_p(\mathcal Q) } \|f_2\|_{\mathcal Y_p(\mathcal Q) }, \qquad p=\frac{2}{2-\eps}. \end{split}
\]
\end{proposition}
\begin{proof} We may assume $A_2< A_1$, otherwise there is nothing to prove. We are allowed to normalize  $A_1=1$.  Fixing now $0<\eps<  1$, so that $1<p <2$, it will suffice to prove the estimate
\begin{equation} \label{intermediate}
|B(f_1,f_2)| \lesssim  (A_2)^{\eps}
\end{equation}
for each pair $f_1 \in \dot{ \mathcal X}_p (\mathcal Q), $    $f_2\in \mathcal Y_p(\mathcal Q) $ with $\|f_1\|_{\dot{ \mathcal X}_p   }=\|f_2\|_{{ \mathcal Y}_p   }=1$ with implied constant depending on dimension only. Let $\lambda \geq 1$ to be chosen later. Using the notation $f_{>\lambda}:=f\cic{1}_{|f|>\lambda}$, we introduce
the decompositions
\[
f_1= g_1+ b_1, \quad b_1 :=\sum_{Q\in \mathcal Q}  \left( (f_1)_{>\lambda } -  \frac{1}{|Q|}\int_Q (f_1)_{>\lambda } \right) \cic{1}_Q, \qquad f_2= g_2+ b_2, \quad b_2 :=  (f_2)_{>\lambda }
\]
which verify the properties
\[
\begin{split}
&g_1 \in \dot{ \mathcal X}_2(\mathcal Q),\; \|g_1\|_{\dot{ \mathcal X}_p } \lesssim 1,\;\|g_1\|_{\dot{ \mathcal X}_2 } \lesssim \lambda^{1-\frac p2}, \qquad  b_1\in \dot{\mathcal X}_1(\mathcal Q), \; \|b_1\|_{\dot{\mathcal X}_1} \lesssim \lambda^{1-p}
\\ & \|g_2\|_{\dot{ \mathcal X}_2 } \lesssim \lambda^{1-\frac p2}, \qquad \|b_2\|_{{\mathcal X}_1}  \lesssim \lambda^{1-p}.
\end{split}
\]
We have used that $b_1$ is supported on the union of the cubes $Q\in \mathcal Q$ and has mean zero on each $Q$, and  therefore  $g_1$ has the same property,  given that $f_1 \in { \dot{ \mathcal X}}_p (\mathcal Q)$.
Therefore
\[
\begin{split}
|B(f_1,f_2)|&\leq |B(b_1,b_2)| + |B(b_1,g_2)| + |B(g_1,b_2)| + |B(g_1,g_2)|\\ & \leq   \|b_1\|_{\dot{\mathcal X}_1 } \|b_2\|_{\mathcal Y_1  } +  \|b_1\|_{\dot{\mathcal X}_1 } \|g_2\|_{\mathcal Y_1  } + \|g_1\|_{\dot{\mathcal X}_1 } \|b_2\|_{\mathcal Y_1  } + A_2\|g_1\|_{ \mathcal Y_2 } \|g_2\|_{\mathcal Y_2 }  \\ &\lesssim \lambda^{2-2p} + 2\lambda^{1-p} + A_2 \lambda^{2-p} \lesssim \lambda^{2-2p}(1+ A_{2}\lambda^{p})
\end{split}\]
which yields \eqref{intermediate} with the choice $\lambda= A_2^{-\frac 1p}$.
\end{proof}

\section{A sparse domination principle for maximal truncations}
\label{S2} We consider families of functions $[K]=\{K_s:s\in \mathbb Z\}$ satisfying
\begin{equation}
\label{assker2}
\begin{split} &
\supp K_s \subset\left\{(x,y)\in \R^d\times \R^d: |x-y|< 2^{s} \right\} ,
 \\
& \|[K]\| :=\sup_{s\in \mathbb Z} 2^{ sd}\sup_{x\in \R^d}  \left(\|K_s(x, \cdot)\|_{\infty} +   \|K_s(\cdot,x)\|_{\infty} \right)<\infty
\end{split}
\end{equation}
and associate to them the   linear operators
\begin{equation}
\label{tis}
T[K]f(x,t_1, t_2): = \sum_{t_1<s\leq t_2}   \int_{\R^d} K_s(x,y)f(y) \, \d y, \qquad x\in \R^d, \, t_1,t_2\in \mathbb Z
\end{equation}
and their sublinear maximal versions
\[
{T_\star}_{t_1}^{t_2}[K]f(x):=\sup_{t_1\leq \tau_1\leq \tau_2\leq  t_2}|T[K]f(x,\tau_1, \tau_2)|, \qquad T_\star[K]f(x) =   \sup_{t_1\leq   t_2}|T[K]f(x,t_1, t_2)|. \]
We assume   that there exists $1<r<\infty$ such that
\begin{equation}
\label{asstrunc}
\|[K]\|_{r,\star} :=  \|T_\star[K]\|_{L^{r}(\R^d)}  <\infty.
\end{equation}
For pairs of bounded measurable functions $t_1,t_2:\R^d\to \mathbb Z$, we also consider the linear operators
\begin{equation}
\label{tisl}
T[K]_{t_1}^{t_2} f(x) := T[K]f(x,t_1(x), t_2(x)), \qquad x\in \R^d.
\end{equation}
\begin{remark} \label{remnottrunc}
From the definition \eqref{tis}, it follows that \[t_1,t_2 \in \mathbb Z, \,t_1\geq t_2\implies  T[K]f(x,t_1, t_2)=0 .\] In consequence, for the linearized versions defined in \eqref{tisl} we have \[
\mathrm{supp}\, T[K]_{t_1}^{t_2} f \subset \{x\in \R^d: t_2(x)-t_1(x)>0\}.
\]
A related word on notation: we will be using linearizations of the type $T[K]_{t_1}^{s_Q}$ and similar, where $s_Q$ is the (dyadic) scale of a (dyadic) cube $Q$. With this we mean we are using the constant function equal to $s_Q$ as our upper truncation function. Finally, we will be using the notations
   $t_2\wedge s_Q$ for the linearizing function $x\mapsto \min\{t_2(x),s_Q\}$ and $t_1\vee s_L$ for the linearizing function $x\mapsto \max\{t_1(x),s_L\}$.
\end{remark}
Given two bounded measurable functions $t_1,t_2$ and a stopping collection $\mathcal Q$ with top $Q$, we define the localized truncated bilinear forms
\begin{equation} \label{therealLambdaQ}\begin{split}
\mathcal Q [K]_{ t_1}^{t_2}(f_1,f_2)& := \frac{1}{|Q|} \left[ \left\l
T[K]^{t_2\wedge s_Q}_{t_1}(f_1\cic{1}_{Q}),f_2\right\r- \sum_{\substack {L\in \mathcal Q\\ L\subset Q} }\left\l
T[K]^{t_2\wedge s_L}_{t_1}(f_1\cic{1}_{L}),f_2\right\r\right]  .\end{split}
\end{equation}
\begin{remark} \label{remonlambdaQ}Note that we have normalized by the measure of $Q$, unlike the definitions in \cite[Section 2]{CoCuDPOu}.
Observe that   as a consequence of the support assumptions in \eqref{assker2} and of the largest allowed scale being $s_Q$, we have
\[
\mathcal Q [K]_{ t_1}^{t_2}(f_1,f_2) = \mathcal Q [K]_{ t_1}^{t_2}(f_1\cic{1}_Q,f_2\cic{1}_{3Q}).
\]
Similarly we remark that $T[K]^{t_2\wedge s_L}_{t_1}(f\cic 1_L)$ is supported on the set $3L\cap\{x\in \R^d: s_L-t_1(x)>0\};$ see Remark \ref{remnottrunc}.\end{remark}

Within the above framework, we have the following abstract theorem.
 \begin{theorem}
\label{theoremABS}
Let $[K]=\{K_s:s\in \mathbb Z\}$ be a family of functions satisfying    \eqref{assker2} and \eqref{asstrunc} above.
Assume that   there exist $1\leq p_1,p_2<\infty$ such that \begin{equation}
\label{lest}
\begin{split}
&\sup_{\substack{\|b\|_{\dot{\mathcal X}_{p_1}(\mathcal Q)}=1 \\ \|f\|_{\mathcal Y_{p_2}(\mathcal Q)}=1}}
\left|\mathcal Q [K]_{ t_1}^{t_2}(b,f)\right|  + \sup_{\substack{\|f\|_{{\mathcal Y}_{\infty}(\mathcal Q)}=1 \\ \|b\|_{\dot{\mathcal X}_{p_2}(\mathcal Q)}=1 }}\left|\mathcal Q [K]_{ t_1}^{t_2}(f,b)\right|=: C_{\mathsf{L}}[K](p_1,p_2)<\infty.
\end{split} \end{equation}
hold uniformly over all bounded measurable functions $t_1,t_2,$  and all stopping collections $\mathcal Q$.
 Then
\begin{equation}
\label{thABSest}
\|   T_\star[K]   \|_{\left(p_1,p_2\right),\mathsf{sparse}}
\lesssim
\|[K]\|_{r,\star} +C_{\mathsf{L}}[K](p_1,p_2).
\end{equation}
\end{theorem}
\begin{proof}The proof   follows     essentially the same   scheme of  \cite[Theorem C]{CoCuDPOu}; for this reason, we limit ourselves to providing an outline of the main steps.
\subsubsection*{Step 1. Auxiliary estimate} First of all, an immediate consequence of the assumptions of the Theorem is that
 the estimate
\begin{equation}
\label{stopiter0}
\left| \mathcal Q [K]_{ t_1}^{t_2}(f_1,f_2)
 \right| \leq C \mathbf{\Theta}_{[K],p_1,p_2} \|f_1\|_{\mathcal Y_
{p_1}(\mathcal Q) } \|f_2\|_{\mathcal Y_{p_2}(\mathcal Q) }  
\end{equation}
where $\mathbf{\Theta}_{[K],p_1,p_2}:= \|[K]\|_{r,\star} +C_{\mathsf{L}}[K](p_1,p_2)  $, holds with
 $C>0$  uniform over bounded measurable functions $t_1,t_2$.   See \cite[Lemma 2.7]{CoCuDPOu}. Therefore,
\begin{equation}
\label{stopiter1}
\left|\left\l
T[K]^{t_2\wedge s_Q}_{t_1}(f_1\cic{1}_{Q}),f_2\right\r\right| \leq C \mathbf{\Theta}_{[K],p_1,p_2} |Q|\|f_1\|_{\mathcal Y_
{p_1}(\mathcal Q) } \|f_2\|_{\mathcal Y_{p_2}(\mathcal Q) } + \sum_{\substack{L \in  \mathcal  Q\\ L \subset Q}}
\left|\left\l
T[K]^{t_2\wedge s_L}_{t_1}(f_1\cic{1}_{L}),f_2\right\r\right|
\end{equation}

\subsubsection*{Step 2. Initialization} The argument begins as follows. Fixing $f_j\in L^{p_j}(\R^d),j=1,2$ with compact support, we may find  measurable functions $t_1,t_2$ which are bounded above and below and a large enough dyadic cube $Q_0$ from one of the canonical $3^d$ dyadic systems such that $\supp f_1\subset Q_0,\supp f_2\subset 3Q_0$ and
\[
\left\l
T_\star[K] f_1 ,|f_2|\right\r \leq 2 \left|\left\l
T[K]^{t_2\wedge s_{Q_0}}_{t_1}(f_1\cic{1}_{Q_0}),|f_2|\right\r\right|
\]
and we clearly can replace $f_2$ by $|f_2|$ in what follows.

\subsubsection*{Step 3. Iterative process} Then, the argument   proceeds via iteration over $k$ of the following construction, which follows from \eqref{stopiter1} and the Calder\'on-Zygmund decomposition and is initialized by taking $\mathcal S_k =\{Q_0\}$ for $k=0$. Given   a disjoint collection of dyadic cubes $Q\in \mathcal S_k$ with the further Whitney property  that \eqref{stop3} holds for $\mathcal S_k$ in place of $\mathcal Q$,
 there exists a further collection of disjoint dyadic cubes $L\in \mathcal S_{k+1}$ such that \begin{itemize} \item  \eqref{stop3} for $\mathcal S_k$ in place of $\mathcal Q$ continues to hold,\item  each subcollection $\mathcal S_{k+1}(Q)=\{L\in \mathcal S_{k+1}: L\subset 3Q\}$ is a stopping collection   with top $Q$, \end{itemize} and for which for all $Q\in \mathcal S_k$ there holds
\begin{equation}
\label{stopiter1tilde}
\begin{split} &\quad 
\left|\left\l
T[K]^{t_2\wedge s_Q}_{t_1}(f_1\cic{1}_{Q}),f_2\right\r\right|\\ & \leq C|Q| \mathbf{\Theta}_{[K],p_1,p_2}\l f_1 \r_{p_1,3Q}  \l f_2 \r_{p_2,3Q} + \sum_{\substack{L \in    \mathcal S_{k+1}(Q)\\ L \subset Q}}
\left|\left\l
T[K]^{t_2\wedge s_L}_{t_1}(f_1\cic{1}_{L}),f_2\right\r\right|.
\end{split}
\end{equation}
 More precisely, $\mathcal S_{k+1}$ is composed by the maximal dyadic cubes $L$ such that
 \begin{equation}
\label{EQ}
 9L\subset
 \bigcup_{Q \in \mathcal S_k} E_Q, \qquad E_Q:=\left\{ x \in 3Q: \max_{j=1,2} \frac{\mathrm{M}_{p_j}(f_j\cic{1}_{3Q}) (x)}{ \langle f_j \r_{p_j,3Q}} > C\right\}
\end{equation}
for a suitably chosen absolute large dimensional constant $C$. This construction,  as well as the Whitney property \eqref{stop3} results into   \begin{equation}
\label{stopiter5}
\left|
Q \cap \bigcup_{L \in \mathcal S_{k+1}} L
\right|=
\left|
Q \cap \bigcup_{Q' \in \mathcal S_{k}: Q' \in N(Q) } E_{Q'}
\right|\leq \frac12 |Q| \qquad \forall Q \in \mathcal S_k, \,k=0,1,\ldots
\end{equation}
guaranteeing that $\mathcal T_k:=\cup_{\kappa=0}^{  k} \mathcal S_\kappa$ is a sparse collection for all $k$.
When $k=\bar k$ is such that $\inf\{ s_{Q}:Q \in \mathcal S_{\bar k}\} < \inf t_1 $, the iteration stops and the estimate
\[
\left\l
T_\star[K] f_1 ,f_2\right\r \lesssim   \mathbf{\Theta}_{[K],p_1,p_2}\sum_{Q \in \mathcal T_{\bar k}} |Q|\l f_1 \r_{p_1,3Q}  \l f_2 \r_{p_2,3Q} \] is reached. This
completes the proof of Theorem \ref{theoremABS}.\end{proof}

\section{Preliminary localized estimates for the truncated forms \eqref{therealLambdaQ}}
\label{S3}
We begin by introducing our notation for the Dini constant of a family of kernels $[K]$ as in \eqref{assker2}. We write
\begin{equation}
\label{NORMS}
\|[K]\|_{\mathsf{Dini}}:=\|[K]\|+  \sum_{j=0}^\infty   \varpi_{j}([K])
\end{equation}
where
\[
\varpi_{j}([K]) := \sup_{s\in \mathbb Z} 2^{ {sd}}\sup_{x\in \R^d  }\sup_{\substack{h\in \R^d\\ |h|<  2^{s-j-3}}}\left(\begin{array}{ll}
  & \|K_s(x,x+\cdot)- K_s(x+h,x +\cdot)\|_{\infty} \\  +&   \|K_s(x+\cdot,x)-K_s(x+\cdot,x+h)\|_{\infty}\end{array} \right)
\]
The estimates contained within the lemmata that follow   are meant to be uniform  over all measurable functions $t_1,t_2$ and all    stopping collections  $\mathcal Q $.
The first one is an immediate consequence of the definitions: for a full proof, see  \cite[Lemma 2.3]{CoCuDPOu}.
\begin{lemma} \label{lemma0} Let $1<r<\infty$. Then
\[
\left|\mathcal Q [K]_{ t_1}^{t_2}(f_1,f_2) \right|\lesssim  \|[K]\|_{r,\star}  \|f_1\|_{\mathcal Y_r(\mathcal Q)} \|f_2\|_{\mathcal Y_{r'}(\mathcal Q)}\]
\end{lemma}
The second one is a variant of \cite[Lemma 3.2]{CoCuDPOu}; we provide a full proof.
\begin{lemma}  \label{bbestimate1}  There holds
\[
\left|\mathcal Q [K]_{ t_1}^{t_2}(b,f)\right|  \lesssim  \|[K]\|_{\mathsf{Dini}}  \|b\|_{\dot{\mathcal X}_1(\mathcal Q)}  \|f\|_{\mathcal Y_1(\mathcal Q) }.
\]
\end{lemma}
\begin{proof} We consider the family $K$ fixed and use the simplified notation $\mathcal Q  _{ t_1}^{t_2}$ in place of $\mathcal Q [K]_{ t_1}^{t_2}$, and similarly for the truncated operators $T[K]$. By horizontal rescaling we can assume $|Q|=1$.
Let $b \in \dot {\mathcal X}_1  $. Recalling the definition \eqref{bell} and using bilinearity of $\mathcal Q_{ t_1}^{t_2}$ it suffices for each stopping cube $R\in \mathcal Q$ to prove that
\begin{equation} \label{twoest}
\begin{split}
\left|\mathcal Q _{ t_1}^{t_2}(b_R,f)\right|  &  \lesssim \|[K]\|_{\mathsf{Dini}} \|b_R\|_1  \|f\|_{\mathcal Y_1 }
\end{split}
\end{equation}
as $ \|b_R\|_1 \lesssim |R| \|b\|_{\dot{\mathcal X}_1 } $,
and conclude by  summing up over the disjoint $R\in \mathcal Q$, whose union is contained in $3Q$.
We may further assume $R\subset Q$; otherwise $\mathcal Q _{ t_1}^{t_2}(b_R,f)=0$. In addition we can assume $f$ is positive, by repeating the same argument below with the real and imaginary, and positive and negative parts of $f$.  Using the definition of the truncated forms \eqref{therealLambdaQ} and the disjointness of $L \in \mathcal Q$,\[
\big|\mathcal Q _{ t_1}^{t_2}(b_R,f) \big|= \big|\left\l
T^{t_2\wedge s_Q}_{t_1}(b_R) - T^{ t_2\wedge s_R}_{t_1}(b_R),f\right\r\big|
=\big| \left\l
T^{t_2\wedge s_Q}_{t_1\vee s_R}(b_R),f\right\r \big|\leq \left\l
{T_\star}_{s_R}^{s_Q}b_R,f\right\r.    \]
Thus, if $R_s$ denotes the cube concentric to $R$ and whose sidelength is $2^{10+s}$, using the support conditions and abbreviating a standard calculation
\[
\begin{split}
\big|\mathcal Q _{ t_1}^{t_2}(b_R,f)\big|\leq  \left\l
{T_\star}_{s_R}^{s_Q}b_R,f\right\r   & \leq \sum_{s=s_R+1}^{s_Q} \int_{R_s} \left| \int_{R} K_{s}(x,y) b_R(y)\, \d y\right|  f(x) \,\d x\\ & \lesssim \|[K]\|_{\mathsf{Dini}}  \|b_R\|_{\dot{\mathcal X}_1}  \sup_{j} \l f \r_{1,R_j}
\end{split}
\]
which is bounded by the right hand side of \eqref{twoest}.
\end{proof}
The third localized estimate is  new. However, its roots lie in the well-known principle that the maximal truncations of a Dini-continuous kernel to scales larger than $s$ do not oscillate too much on a ball of radius $2^s$, see \eqref{osclemma}. This was recently employed, for instance, in \cite{FSZK,HRT,Lac2015}.
\begin{lemma}  \label{bbestimate2}  There holds
\[
\left|\mathcal Q [K]_{ t_1}^{t_2}(f,b)\right|  \lesssim \left( \|[K]\|_{\mathsf{Dini}} \vee \|K\|_{r,\star}\right) \|f\|_{\mathcal Y_\infty(\mathcal Q)} \|b\|_{ \mathcal X_1(\mathcal Q)}.
\]
\end{lemma}

\begin{proof}
We use similar notation as in the previous proof and again we rescale to  $|Q|=1$,
and  work with positive $b \in \mathcal X_1$. We can of course assume that $\mathrm{supp}\, f\subset Q $. We begin by removing an error term; namely, referring to notation \eqref{stop1},
if $$
b_{\mathsf{o}} =\sum_{R\not\in \mathsf{c}\mathcal Q} b_R
$$
then
\begin{equation}
\label{bo}
\left|\mathcal Q [K]_{ t_1}^{t_2}(f,b_{\mathsf{o}})\right| \leq   \left \l
|Tf(\cdot,s_Q-1,s_Q)| ,b_{\mathsf{o}}\right\r    \lesssim \|[K]\|_{}  \|b\|_{\mathcal X_1} \|f\|_{\mathcal Y_\infty}
\end{equation}
The first inequality holds because $\mathrm{dist}(\mathrm{supp} f,\mathrm{supp}\,b_{\mathsf{o}}) > 2^{s_Q-1}  $, so at most the $s_Q$ scale may contribute, and in particular no contribution comes from cubes $L\subsetneq Q$. The second inequality is a trivial estimate, see \cite[Appendix A]{CoCuDPOu} for more details.
Thus we may assume $b_R=0$ whenever $R\not\in \mathsf{c}\mathcal Q$.
We begin the main argument by fixing $R \in \mathsf{c}\mathcal Q$. Then by support considerations
\[\left \l
T^{  t_2\wedge s_L}_{t_1}(f\cic{1}_{L}),b_R\right\r \neq 0 \implies L\in N(R). \]
Similarly,    \[
 \left \l
T^{  t_2\wedge s_R}_{t_1}f ,b_R \right\r=\left \l T^{  t_2\wedge s_R}_{t_1}(f\cic{1}_{\mathsf{sh}\mathcal Q}) ,b_R \right\r=\sum_{\substack{L\in N(R)\\ L\subset Q}}
 \left \l T^{ t_2\wedge  s_R}_{t_1}(f\cic{1}_{L}),b_R\right\r.   \]
In fact,  using \eqref{stop1} we learn that $\mathrm{dist}(\mathsf{sh} \mathcal Q,R) > 2^{s_R}$, whence the first equality.
Therefore,   subtracting and adding the last display to obtain the second equality,
 \[\begin{split}
 &\quad
\mathcal Q _{ t_1}^{t_2}(f,b_R)=
\left \l  T^{t_2\wedge s_Q}_{t_1}f,b_R\right\r - \sum_{\substack{L\in N(R)\\ L\subset Q}}\left \l  T^{ t_2\wedge  s_L}_{t_1}(f\cic{1}_{L}),b_R\right\r \\ & =   \left\l  T^{t_2\wedge s_Q}_{t_1\vee s_R}f,b_R\right\r  -  \sum_{\substack{L\in N(R)\\ L\subset Q}} \mathrm{sign}(s_L-s_R) \left \l T^{ t_2\wedge (  s_L \vee s_R)}_{t_1\vee (s_L \wedge s_R)}(f\cic{1}_{L}),b_R\right\r. \\
\end{split}
 \] Now, the summation in the above display is then bounded in absolute value by
\[
\sum_{L\in N(R)} \left \l {T_\star}_{s_L \vee s_R}^{s_L \wedge s_R} (f\cic 1_L),b_R \right\r \lesssim \|[K]\| \sum_{L\in N(R)} \|b_{R}\|_1 \l f \r_{1,L} \lesssim  |[K]\|_{\mathsf{Dini}}  \|b_R\|_1 \|f\|_{\mathcal Y_\infty},
\]
using that $|s_L-s_R|\leq 8$ whenever $L\in N(R)$. Therefore when $R\in \mathsf{c}\mathcal Q$
\begin{equation}
\label{interm1}
\begin{split}  &\quad
\left|\mathcal Q _{ t_1}^{t_2}(f,b_R)\right|\leq
\left| \left\l  T^{t_2\wedge s_Q}_{t_1\vee s_R}f,b_R\right\r\right|  + C \|[K]\|_{\mathsf{Dini}}  \|b_R\|_1 \|f\|_{\mathcal Y_\infty}
\end{split}
\end{equation}
with absolute constant $C$. Now, define the function
\[
F(x) = \begin{cases}\displaystyle \sup_{s_R\leq \tau_1\leq \tau_2\leq s_Q} |Tf(x,\tau_1,\tau_2)| & x\in R \in \mathsf{c}\mathcal Q \\ 0 & x \not \in \bigcup_{R\in \mathsf{c} \mathcal Q} R
\end{cases}
\]
and notice that $|T^{t_2\wedge s_Q}_{t_1\vee s_R}f|\leq F$ on $R\in \mathsf{c}\mathcal Q$. Since $b$ is positive, using \eqref{bo},
summing \eqref{interm1} over $R\in \mathsf{c}\mathcal Q$ and using that this is a pairwise disjoint collection, we obtain that
\begin{equation}
\label{interm2}
\begin{split} &\quad
|\mathcal Q _{ t_1}^{t_2}(f,b)|\leq |\mathcal Q _{ t_1}^{t_2}(f,b_{\mathsf{o}})| + \sum_{R\in \mathsf{c}\mathcal Q} |\mathcal Q _{ t_1}^{t_2}(f,b_R)|\\ &\leq  C \|[K]\|_{\mathsf{Dini}}  \|b\|_{\mathcal Y_1} \|f\|_{\mathcal Y_\infty}
+\sum_{R\in \mathsf{c} \mathcal Q}  \left|\left\l  T^{t_2\wedge s_Q}_{t_1\vee s_R}f,b_R\right\r\right|  \\ & \leq C  \|[K]\|_{\mathsf{Dini}}  \|b\|_{\mathcal X_1} \|f\|_{\mathcal Y_\infty}
+\sum_{R\in  \mathsf{c} \mathcal Q} \left\l F,b_R\right\r= C \|[K]\|_{\mathsf{Dini}}  \|b\|_{\mathcal X_1} \|f\|_{\mathcal Y_\infty} + \left\l  F,b\right\r.
\end{split}
\end{equation}
Therefore, we are left with bounding $\l F,b \r $. This is actually done using both the $L^r$ estimate and the Dini cancellation condition. In fact, decompose
\[
b= g + z, \qquad g=\sum_{R\in  \mathsf{c}\mathcal Q} g_R: = \sum_{R\in  \mathsf{c}\mathcal Q} \l b \r_{1,R}\cic{1}_R, \qquad z=\sum_{R\in \mathsf{c} \mathcal Q} z_R: = \sum_{R\in  \mathsf{c}\mathcal Q} \left(b- \l b\r_{1,R}\right)\cic{1}_R
\]
so that
\[
\|g\|_{{{\mathcal Y}}_\infty}\leq  \|b\|_{{{{\mathcal X}}_1}}, \qquad
\|z\|_{{\dot{\mathcal X}}_1}\leq 2 \|b\|_{{{{\mathcal X}}_1}}\]
Then
\begin{equation}
\label{interm3}
\l F,g\r \leq \l T_\star f, g\r \leq \|[K]\|_{r,\star}  \|f\|_{r} \|g\|_{r'} \leq \|[K]\|_{r,\star}  \|g\|_{{{\mathcal Y}}_\infty} \|f\|_{{{\mathcal Y}}_\infty} \leq \|[K]\|_{r,\star}  \|f\|_{{{\mathcal Y}}_\infty} \|b\|_{{{\mathcal X}}_1}
\end{equation}
and we are left to control $|\l F, z\r |$.
We recall from \cite[Lemma 2.1]{HRT} the inequality
\begin{equation}
\label{osclemma}
|Tf(x,\tau_1,\tau_2)-Tf(\xi,\tau_1,\tau_2)| \lesssim \|[K]\|_{\mathsf{Dini}} \sup_{s\geq s_R} \l f \r_{1,R_s}, \qquad x,\xi \in R, \tau_2
\geq \tau_1\geq s_R
\end{equation}
where   $R_s$ is the cube concentric  with $R$ and sidelength $2^s$, whence for suitable absolute constant $C$
\[
F  (x)\leq F(\xi)  +  C \|[K]\|_{\mathsf{Dini}} \|f\|_{\mathcal Y_1}, \qquad x,\xi \in R
\]
and taking  averages there holds
\[
\sup_{x\in R}|F(x) - \l F \r_{1,R}|\lesssim   \|[K]\|_{\mathsf{Dini}} \|f\|_{\mathcal Y_1}.
\]
Finally, using the above display and the fact that each $z_R$ has zero average and is supported on $R$,
\begin{equation}
\label{interm4}
\begin{split}
|\l F, z\r| &\leq\sum_{R\in \mathcal Q } |\l  F,z_R\r| =
\sum_{R\in\mathsf{c} \mathcal Q } |\l  F-\l F \r_{1,R}\cic{1}_R,z_R\r|  \lesssim \|[K]\|_{\mathsf{Dini}}  \|f\|_{\mathcal Y_1}\sum_{R\in\mathsf{c} \mathcal Q } \|z_R\|_1 \\ &\leq   \|[K]\|_{\mathsf{Dini}}  \|f\|_{\mathcal Y_\infty}  \|b\|_{\mathcal X_1}
\end{split}
\end{equation}
and collecting \eqref{interm2}, \eqref{interm3} and \eqref{interm4} completes the proof of the Lemma.
\end{proof}
By Proposition \ref{interprop} applied to the forms $\mathcal Q _{ t_1}^{t_2}[K]$, we may interpolate  the bound of Lemma \ref{bbestimate1} with the one of   Lemma \ref{lemma0} with $r=2$. A similar but easier procedure allows to interpolate Lemma \ref{bbestimate2} with   Lemma \ref{lemma0} with $r=2$. We summarize the result of such interpolations in the following lemma.
\begin{lemma}  \label{ppestimate}   For $0\leq \eps\leq 1$ and $p=\frac{2}{2-\eps}$ there holds
\[
C_{\mathsf{L}}[K](p,p)\lesssim \left( \\|[K]\|_{\mathsf{Dini}} \vee \|[K]\|_{2,\star} \right)^{1-\eps} \left(\|[K]\|_{2,\star}\right)^{\eps}\]
where $C_{\mathsf{L}}[K](p_1,p_2)$ is defined in \eqref{lest}.
\end{lemma}
\begin{remark}[Calder\'on-Zygmund theory]
Let $T$ be an $L^2(\R^d)$-bounded singular integral operator with  Dini-continuous kernel $K$. Then its maximal truncations obey the estimate
\begin{equation}
\label{pointCZ}
T_{\star} f(x):=
\sup_{\delta>0}\left| \int_{\delta<|h|<\frac1\delta} K(x,x+h) f(x+h) \, \d h\right|\lesssim \mathrm{M} f(x) + T_{\star}[K] f(x),
\end{equation} 
with  the family $[K]:=\{K_s:s\in \mathbb Z\}$  defined by
\[
K_s(x,x+h) := K(x,x+h) \psi (2^{-s}h ), \qquad x,h \in \R^d
\]
where the smooth  radial function $\psi$ satisfies
\[
\supp \psi\subset\{h\in \R^d:2^{-2}<|h|<1\}, \qquad \sum_{s\in \mathbb Z} \psi (2^{-s}h )=1, \, h\neq 0.
\]
We know from classical theory  \cite[Ch.\ I.7]{Stein} that $\|[K]\|_{2,\star}\lesssim \|T\|_{L^2(\R^d)}+ \|K\|_{\mathsf{Dini}} $. Therefore, in consequence of \eqref{pointCZ} and  of the bound $\|\mathrm M\|_{(1,1),\mathsf{sparse}}\lesssim 1$, an application of Theorem \ref{theoremABS} in conjunction with Lemmata \ref{bbestimate1} and \ref{bbestimate2} yields that
\[
\|T_\star\|_{(1,1),\mathsf{sparse}} \lesssim \|T\|_{L^2(\R^d)}+ \|K\|_{\mathsf{Dini}}.
\]
This  is a well-known result. The dual pointwise version was  first obtained  in this form in \cite{HRT} quantifying the initial result of Lacey \cite{Lac2015}; see also \cite{Ler2015}. An extension to multilinear operators   with less regular kernels was recently obtained in \cite{Li1}.
\end{remark}
\section{Proof of Theorem \ref{theoremRK}} \label{S4}
In this section,
we will prove Theorem \ref{theoremRK} by appealing to Theorem \ref{theoremABS} for the family $[K]=\{(x,y)\mapsto K_s(x-y):s\in \mathbb Z\}$ of \eqref{assker}. First of all, we notice that the assumption \eqref{assker2} is a direct consequence of \eqref{assker}. It is known from e.g.\ \cite{DR} (and our work below actually reproves this) that, with reference to \eqref{deftstar}
\[
\|T_\star[K]\|_{L^2(\R^d)} \lesssim 1,
\]
which is assumption \eqref{asstrunc} with $r=2$. Therefore, for an application of Theorem \ref{theoremABS} with
\[p_1=p_2=p=\frac{2}{2-\eps}, \qquad 0<\eps<1 \]
 we are left with verifying the corresponding stopping estimates \eqref{lest} hold with $C_\mathsf{L}[K](p,p) \lesssim \eps^{-1}$.
We do so by means of     a     Littlewood-Paley decomposition, as follows. Let $\varphi$ be a smooth radial function on $\R^d$ with support in a sufficiently small ball containing the origin, having mean zero and such that
\[
  \sum_{k=-\infty}^{\infty} \widehat{ \varphi_k}( \xi) =1, \quad \forall \xi \neq 0, \qquad \varphi_{k}(\cdot):= 2^{-kd} \varphi({ 2^{-k} }\cdot).    
\]
Also define
\begin{equation} \label{LP}
\phi_{k}(\cdot) :=  \sum_{\ell \geq k}  {\varphi_\ell(\cdot)}, 
\qquad K_{s,0}= K_s *\phi_s ,\qquad K_{s,j}  = \sum_{\ell=\Delta(j-1)+1}^{\Delta j}K_{s}*\varphi_{s-\ell}, \,j\geq 1. \end{equation}
for some large integer $\Delta$ which will be specified during the proof. Unless otherwise specified, the implied constants appearing below are independent of $\Delta$ but may depend on $\alpha>0$ from \eqref{assker} and on the dimension.
Note that  $K_{s,j}$ are supported in $\{|x|<2^s\}$.
Define now for all $j\geq 0$
\[
[K^{j}]=\{ (x,y)\mapsto K_{s,j}(x-y):s \in \mathbb Z\}
\]
and note that, with unconditional convergence
\begin{equation} \label{LP2}
K_s(y) = \sum_{j=0}^\infty K_{s,j}(y), \qquad y \in \mathbb \R^d.
\end{equation}
The following computation is carried out in \cite[Section 3]{HRT}. \begin{lemma} \label{kernorms1} There holds
\begin{equation}
\label{varpi}
\varpi_{\ell}([K^j]) \lesssim \min \{1, 2^{\Delta j-\ell}\}
\end{equation}
 and as a consequence
  $
\|[K^{j}]\|_{\mathsf{Dini}} \lesssim  1+\Delta j$ for all $ j\geq 0.
$
\end{lemma}
It is also well-known that
\begin{equation}
\label{untruncest}
\sup_{t_1, t_2\in \mathbb Z} \left\| f\mapsto T[K^j]f(\cdot,t_1,t_2)=\sum_{t_1<s\leq t_2} K_{s,j}*f \right\|_{L^2(\R^d)} \lesssim 2^{-\alpha \Delta(j-1)};
\end{equation}
however, we need a stronger estimate on the pointwise maximal truncations, which is implicit in \cite{DR}.
\begin{lemma} \label{kernorms3}  There holds\[
\|[K^{0}]\|_{2,\star} \lesssim 1, \qquad \displaystyle\|[K^{j}]\|_{2,\star} \lesssim 2^{- {\frac\alpha2 \Delta (j-1)}}, \quad j\geq 1.\]
\end{lemma}
\begin{proof} Let $\beta$ be a smooth compactly supported function on $\R^d$ normalized to have $\l \beta,1\r=1$, and write $\beta_s(\cdot)=2^{-sd} \beta(2^{-s}\cdot)$. By usual arguments it suffices to estimate
the $L^2(\R^d) $ operator norm of
\[
f\mapsto \sup_{t_1\leq s\leq t_2} T[K^j]f(\cdot, s,t_2)
\]
uniformly over $t_1,t_2 \in \mathbb Z$.
We then have \begin{equation}
\label{pfmt1}
\begin{split}
T[K^j]f(\cdot, s,t_2)&= \beta_{s} *\big(\sum_{t_1<k\leq t_2}K_k * (\phi_{k-\Delta j}-\phi_{k-\Delta(j-1)})\big)*f\\
&- \beta_{s}*\big(\sum_{t_1<k\le s}K_k * (\phi_{k-\Delta j}-\phi_{k-\Delta(j-1)})\big)*f\\
&+ (\delta- \beta_{s})*\big(\sum_{s<k \leq t_2}K_k * (\phi_{k-\Delta j}-\phi_{k-\Delta(j-1)})\big)*f=:I_{1,s}+I_{2,s}+I_{3,s},
\end{split}
\end{equation}
For $I_{1,s}$, by \eqref{untruncest} we have
\begin{equation}
\label{pfmt2}
\left\|\sup_{t_1\leq s \leq t_2} |I_{1,s}|\right\|_{2}\lesssim  \left\|\mathrm{M}\big(T[K^j] f(\cdot, t_1,t_2)\big)\right\|_{2}\lesssim  2^{-\alpha \Delta(j-1)}\|f\|_{2}.
\end{equation}
 Next we estimate the second and third contribution in \eqref{pfmt1}. We have, using the third assertion in  \eqref{assker}, that
\begin{align*}
&\quad \big|\sum_{t_1<k\leq s}\hat \beta(2^s \xi) \hat K_k(\xi)(\hat \phi(2^{k-\Delta j}\xi)-\hat \phi(2^{k-\Delta (j-1)}\xi))\big|\\
&\lesssim\big|\sum_{t_1<k\leq s}\min \{1, |2^s \xi|^{-1}\}\cdot \min\{|2^k \xi|^\alpha, |2^k \xi|^{-\alpha}\}\cdot \sum_{\ell=\Delta(j-1)}^{\Delta j} \min\{|2^{k-\ell}\xi|, |2^{k-\ell}\xi|^{-1}\}\big|\\
&\lesssim
\begin{cases}
  2^{-\alpha\Delta(j-1)} |2^s \xi|^{-1}, & |2^s \xi|>1,\\
2^{-\Delta(j-1)} |2^s \xi|, & |2^s\xi|\le 1
\end{cases}
\end{align*}
A similar computation reveals
\begin{align*}
\big|\sum_{k\ge s}&(1- \hat \beta(2^s \xi))\hat K_k(\xi) ( \hat \phi(2^{k-\Delta j}\xi)-\hat \phi(2^{k-\Delta(j-1)}\xi))\big|\le 2^{-\alpha \Delta (j-1)/2} \min\{|2^s\xi|, |2^s\xi|^{-\alpha/2}\}.
\end{align*}
Thus by Plancherel, for $m=2,3$ we have  \begin{equation}
\label{pfmt3}
 \left\|\sup_{t_1\leq s\leq t_2} |I_{m,s}| \right\|_{2}\le  \left\|\big(\sum_{s=t_1}^{t_2} |I_{m,s}|^2\big)^{\frac 12}  \right\|_2\lesssim 2^{-\alpha\Delta (j-1)/2}\|f\|_{2}
\end{equation}
and the proof of the Lemma is completed by putting together \eqref{pfmt1}--\eqref{pfmt3}.
\end{proof}

We are now ready to verify the assumptions \eqref{lest} for the truncated forms $\mathcal Q[K]_{t_1}^{t_2}$ associated to a family $[K]$ satisfying the assumptions \eqref{assker}.
By virtue of Lemma \ref{kernorms1} and \ref{kernorms3}, Lemma \ref{ppestimate} applied to the families $[K^j]$ for the value $\Delta=2\eps^{-1}\alpha^{-1}$   yields that
\[
\begin{split}  C_{\mathsf L }[K^j](p,p)  \lesssim \left( \|[K^j]\|_{\mathsf{Dini}} \vee |[K^j]\|_{2,\star}\right)^{1-\eps} \left(\|[K^j]\|_{2,\star}\right)^{\eps}    \lesssim   (1+\Delta j)^{1-\eps}2^{- \frac{ \alpha}{2}  \Delta (j-1) \eps}\lesssim \eps^{-1}(1+ j)2^{-j}.  \end{split}
\]
Therefore using linearity in the kernel family $[K]$ of the truncated forms $\mathcal Q_{t_1}^{t_2}[K] $ and the decomposition \eqref{LP}-\eqref{LP2}
\begin{equation}
\label{finalest}
 C_{\mathsf L }[K](p,p)  \leq \sum_{j=0}^{\infty} C_{\mathsf L }[K^j](p,p) \lesssim \eps^{-1}
\end{equation}
which, together with the previous observations, completes the proof of Theorem \ref{theoremRK}.

\section{Extension to vector-valued functions} \label{svv}
In this section, we suitably extend the abstract domination principle  Theorem \ref{theoremABS} to (a suitably defined) $\mathbb F^n$-valued extension, with $\mathbb F\in \{\R,\mathbb C\}$, of the singular integrals of Section \ref{S2}. In fact, the $\mathbb C^{n}$-valued case  can be recovered by suitable interpretation of the  $\R^{2n}$-valued one; thus, it suffices to consider $\mathbb F=\R$.  \subsection{Convex body domination}
Let $1\leq p<\infty$. To each $f\in L^p_{\mathrm{loc}}(\R^d; \R^n)$ and   each cube $Q$ in $\R^d$,  we associate the closed convex symmetric subset of $\R^n$ 
\begin{equation} \label{convexbody}
\l f \r_{p,Q}:= \left\{\frac{1}{|Q|}\int_Q f \varphi \, \d x: \varphi\in \Phi_{p'}(Q)  \right\}\subset \R^n, 
\end{equation}
where we used the notation $\Phi_{q}(Q):=\{\varphi: Q\to \R, \l\varphi\r_{q,Q}\leq 1 \}$.  It is easy to see that
\[
\sup_{\xi \in \l f \r_{p,Q} } |\xi | \leq \l |f|_{\R^n} \r_{p,Q}
\]
where $\l \cdot \r_{p,Q}$ on the right hand side is being interpreted in the usual fashion. A slightly less obvious fact that we will use below is recorded in the following simple lemma, which involves the notion of \emph{John ellipsoid} of a closed convex symmetric set $K$. This set, which we denote by $\mathcal E_K$, stands for  the solid ellipsoid of largest volume contained in $K$; in particular, the John ellipsoid of $K$ has the property that
\begin{equation}
\label{jell}
\mathcal E_K \subset K \subset \sqrt{n}\mathcal E_K
\end{equation}
where, if  $A\subset \R^n$ and $c\geq 0$, by $cA$ we mean the set $\{ca: a \in A\}$. {We also apply this notion in the degenerate case as follows: if the linear span of $K$ is a $k$-dimensional subspace $V$ of $\R^n$, we denote by $\mathcal E_K$ the solid ellipsoid of largest $k$-dimensional volume contained in $K$. In this case, \eqref{jell} holds with $\sqrt{k}$ in place of $\sqrt{n}$, but then it also holds as stated, since $k\leq n$ and $\mathcal E_K$ is also convex and symmetric.}

\begin{lemma}
\label{simplelemma} Let $f=(f_1,\ldots,f_n)\in L^p_{\mathrm{loc}}(\R^d; \R^n)$ and suppose that $\mathcal E_{\l f\r_{p,Q}}=B_1$, where $B_{\rho}=\{a\in \R^n: |a|_{\R^n}\leq \rho\}$. Then  
\[
\sup_{j=1,\ldots, N} \l f_j \r_{p,Q} \leq \sqrt{n}.
\]
\end{lemma}
\begin{proof} By definition of $\l \cdot \r_{p,Q} $ for scalar functions, 
\[
\l f_j \r_{p,Q}= \sup_{\varphi \in   \Phi_{p'}(Q)}  \frac{1}{|Q|}\int_Q f_j \varphi \, \d x =  \frac{1}{|Q|}\int f_j \varphi_\star \, \d x   \]
for a suitable $\varphi_\star\in \Phi_{p'}(Q) $. Thus $ \l f_j \r_{p,Q}$ is the $j$-th component of the vector
\[f_{\varphi_\star}=\frac{1}{|Q|}\int_Q f_j \varphi_\star \, \d x \in \l f \r_{p,Q}.
\]
By \eqref{jell}, and in consequence of the assumption,  $f_{\varphi_\star}\in B_{\sqrt n}$, which proves the assertion.\end{proof}
We now define the sparse $(p_1,p_2)$ norm of a linear operator $T$  mapping the space $L^\infty_0(\R^d; \R^n)$  into  locally integrable, $\R^n$-valued functions, as the least constant $C>0$ such that for each pair $f_1,f_2\in L^\infty_0(\R^d; \R^n)$ we may find a $\frac12$-sparse collection $\mathcal S$ such that 
\begin{equation}
\label{sparsevv}
|\l Tf_1,f_2 \r| \leq C \sum_{Q\in  \mathcal{S}} |Q| \l f_1 \r_{p_1,Q}  \l f_2 \r_{p_2,Q}.
\end{equation}
We interpret the rightmost product in the above display   as  the right endpoint of the Minkowski product $AB=\{\l a,b\r_{\R^n}: a\in A, b\in B\}$ of the  closed convex symmetric sets $A,B\subset \R^n$, which is a closed symmetric interval.
We use the same familiar notation $\|T\|_{(p_1,p_2),\mathsf{sparse}}$ for such norm.  

Within such framework, we have the following extension of Theorem \ref{theoremABS}.
 \begin{theorem}
\label{theoremABSvv}
Let $[K]=\{K_s:s\in \mathbb Z\}$ be a family of real-valued functions satisfying    \eqref{assker2}, \eqref{asstrunc}, and \eqref{lest} for some $1<r<\infty, 1\leq p_1,p_2<\infty$. Then the $\R^n$-valued extension of the  linearized truncations $T[K]^{t_2}_{t_1}$ defined  in  \eqref{tisl} admits a $(p_1,p_2)$ sparse bound, namely
\begin{equation}
\label{thABSest}
\left\| T[K]_{t_1}^{t_2} \otimes \mathrm{Id}_{\R^n} \right\|_{\left(p_1,p_2\right),\mathsf{sparse}}
\lesssim
\|[K]\|_{r,\star} +C_{\mathsf{L}}[K](p_1,p_2)
\end{equation}
with implicit constant possibly depending on $r,p_1,p_2$ and the dimensions $d,n$ only, and in particular uniform over bounded measurable truncation functions $t_1,t_2$.
\end{theorem}
\begin{remark}The objects \eqref{convexbody} for $p=1$ have been introduced  in this context by Nazarov, Petermichl, Treil and Volberg \cite{NPTV}, where  sparse domination of vector valued singular integrals by  the Minkowski sum of convex bodies \eqref{convexbody}   is employed towards matrix-weighted norm inequalities.  In \cite{CuDPOu2}, a similar result, but in the dual form \eqref{sparsevv} with $p_1=p_2=1$ is proved for dyadic shifts via a different iterative technique which is a basic version of the proof of Theorem \ref{theoremABS}. Subsequent  developments in vector valued sparse domination include the sharp estimate for the dyadic square function \cite{HPV}. The usage of exponents $p>1$ in \eqref{convexbody},  necessary to effectively tackle    rough singular integral operators, is a novelty of this paper. 
\end{remark}
\subsection{Matrix-weighted norm inequalities} We now detail an application of Theorem \ref{thABSest} to matrix-weighted norm inequalities for    maximally truncated, rough singular integrals. In particular,  Corollary \ref{theoremW} from the introduction is a particular case of Theorem \ref{theoremW} below.

 The classes of weights we are concerned with are the following.  A pair of matrix-valued weights 
 $W,V\in L^1_{\mathrm{loc}}(\R^d; \mathcal L(\R^n))$ is said to satisfy the (joint) \emph{matrix $A_2$ condition} if  \begin{equation*}
[W,V]_{A_2}:=\sup_{Q}\left\|\sqrt{W_Q} \sqrt{V_Q}\right\|_{\mathcal L(\R^n)}^2<\infty
\end{equation*}
supremum being taken over all cubes $Q\subset \R^d$  and 
\[
W_Q:= \frac{1}{|Q|}\int_{Q} W(x) \, \d x  \in \mathcal L(\R^n).
\]
We simply write $[W]_{A_2}:=[W,W^{-1}]_{A_2}$. We further  introduce a directional matrix $A_\infty $ condition, namely 
\[
[W]_{A_\infty}:=   \sup_{\xi\in S^{n-1}} [ \l W \xi,\xi\r_{\R^n}]_{A_\infty} \leq \sup_{\xi\in S^{n-1}} [ \l W \xi,\xi\r_{\R^n}]_{A_2} \leq [W]_{A_2}.
\]
where the second inequality is the content of  \cite[Lemma 4.3]{NPTV}.
\begin{theorem} \label{theoremW}Let  $W,V\in L^1_{\mathrm{loc}}(\R^d; \mathcal{L}(\R^n))$ be a pair of matrix weights, and $T_{\Omega,\delta}$ be defined by \eqref{Tomegas}, with in particular $\|\Omega\|_\infty\leq 1$. Then
\begin{equation}
\label{theoremWest}
\left\| \sup_{\delta>0} \left| W^{\frac12} T_{\Omega,\delta} (V^{\frac12}f ) \right|_{\R^n} \right\|_{L^2(\R^d)} \lesssim \max\{[W]_{A_\infty}, [V]_{A_\infty} \} \sqrt{[W,V]_{A_2} [W]_{A_\infty} [V]_{A_\infty} }\left\|f \right\|_{L^2(\R^d; \R^n)}.
\end{equation} 
 \end{theorem} 
We now explain how an application of Theorem \ref{theoremABSvv} reduces Theorem \ref{theoremW} to a weighted square function-type estimate for  convex-body valued sparse operators. First of all, fix $f$ of unit norm in $L^2(\R^d; \R^n)$. We may then find $g$ of unit norm  in $L^2(\R^d; \R^n)$ and bounded measurable functions $t_1,t_2$ such that the left hand side of \eqref{theoremWest} is bounded by twice the sum of
\begin{equation}\label{eq:matrixMainTerm}
 \left|\left\langle T[K]_{t_1}^{t_2} \otimes \mathrm{Id}_{\R^n} (V^{\frac12}f), W^{\frac12} g \right\rangle\right|
\end{equation}
and \[
  \left\|x\mapsto\sup_{Q\owns x}\left\langle \left| W^{\frac12}(x)V^{\frac12}f\right|_{\R^n} \right\rangle_Q\right\|_{L^2(\R^d)}
  \leq \left\|x\mapsto\sup_{Q\owns x} \left| W^{\frac12}(x)\l V\r_Q^{\frac12}\right|_{\mathcal L(\R^n)} \left\langle \left|
      \l V\r_Q^{-\frac12}V^{\frac12}f\right|_{\R^n} \right\rangle_Q\right\|_{L^2(\R^d)},
\]
where $[K]$ is a suitable decomposition satisfying \eqref{assker}. {The latter expression is (the norm of) a two-weight version of the matrix weighted maximal function of Christ and Goldberg \cite{CG}. In the one weight case when $V=W^{-1}\in A_2$, its boundedness has been proved in \cite{CG} and quantified in \cite{IKP15}, which contains the explicit bound \[c_{d,n}[W]_{A_2}\|f\|_{L^2(\R^d;\R^n)}\] and the implicit improvement \[c_{d,n}[W]_{A_2}^{1/2}[W^{-1}]_{A_\infty}^{1/2}\|f\|_{L^2(\R^d;\R^n)},\] where $c_{d,n}$ is a dimensional constant. A straightforward modification of the same argument, using the splitting on the right of the previous display, gives the bound \[[W,V]_{A_2}^{1/2}[V]_{A_\infty}^{1/2}\|f\|_{L^2(\R^d;\R^n)}\] in the two weight case. Roughly speaking, the first factor is controlled by the two weight $A_2$ condition and the second one by the $A_\infty$ property of $V$.}

 By virtue of the localized estimate \eqref{finalest} for $[K]$, an application of   Theorem \ref{theoremABSvv} tells us that  \eqref{eq:matrixMainTerm} is bounded by $C/\eps$ times  a sparse sublinear form as in  \eqref{sparsevv} with $p_1=p_2=1+\eps$, $f_1=V^{\frac12}f$ and $f_2=W^{\frac12} g$ for all $\eps>0$. Finally, we gather that
 \begin{equation}
\label{weightedMAIN}\begin{split} &\quad 
\left\| \sup_{\delta>0} \left| W^{\frac12} T_{\Omega,\delta} (V^{\frac12}f ) \right|_{\R^n} \right\|_{2}\\
 & \lesssim  \sqrt{[V,W]_{A_2}\max\{[W]_{A_\infty}, [V]_{A_\infty}\}} + \inf_{\eps>0} \sup_{\mathcal S}   \frac{1}{\eps}   \sum_{Q\in  \mathcal{S}} |Q| \l V^{\frac12}f\r_{1+\eps,Q}  \l W^{\frac12} g \r_{1+\eps,Q}
  \end{split}
\end{equation}
where the supremum is being taken over $\frac12$-sparse collections $\mathcal S$,
and the proof of Theorem \ref{theoremW} is completed by the following proposition.

\begin{proposition} \label{matrixform} The estimate 
\[
\inf_{\eps>0}    \frac{1}{\eps}   \sum_{Q\in  \mathcal{S}} |Q| \l V^{\frac12}f\r_{1+\eps,Q}  \l W^{\frac12} g \r_{1+\eps,Q} \lesssim \max\{[W]_{A_\infty}, [V]_{A_\infty} \} \sqrt{[W,V]_{A_2} [W]_{A_\infty} [V]_{A_\infty} }\]
  holds uniformly over all $f,g$ of unit norm in ${L^2(\R^d;\R^n)}$  and all $\frac12$-sparse collections $\mathcal S $.
\end{proposition}

\begin{proof}
There is no loss in generality with assuming that the sparse collection $\mathcal S$ is a subset of a standard dyadic grid in $\R^d$, and we do so. Fix $\eps>0$.   By standard reductions, we have that
\begin{equation}
\label{S8eq1} \sup_{\|f_j\|_{L^2(\R^d;\R^n)}=1}
\sum_{Q\in  \mathcal{S}} |Q| \l V^{\frac12}f_1\r_{1+\eps,Q}  \l W^{\frac12} f_2 \r_{1+\eps,Q} \lesssim \sqrt{[W,V]_{A_2}}  \sup_{\|f_j\|_{L^2(\R^d;\R^n)}=1} \|S_{V,\eps} f_1\|_2 \|S_{W,\eps} f_2\|_2  \end{equation}
having defined the square function
\[
S_{W,\eps} f^2 =  \sum_{Q\in \mathcal S} \left\l \left|(W_{Q})^{-\frac12} W^{\frac12}\right|\left|f\right|_{\R^n} \right\r_{1+\eps,Q}^2\cic{1}_Q.
\]
Now, if \[ \eps<2^{-10}t_W, \qquad 
t_{W}:=(2^{d+N+1}1+[W]_{A_\infty})^{-1}, \qquad  p:=\frac{1+2t}{(1+\eps) (1+t)}\in (1,2) 
\]as a result of the sharp reverse H\"older inequality and of the Carleson embedding theorem there holds  \[
\|S_{W,\eps} f\|_2^2 \lesssim  \sum_{Q \in S}|Q| \l \left|f\right|_{\R^n}  \r_{\frac2p, Q}^2 \lesssim (p')^p \|\left|f\right|_{\R^n}\|_{2}^2 \lesssim [W]_{A_\infty}\|\left|f\right|_{\R^n}\|_{2}^2\qquad 
\]
cf.\ \cite[Proof of Lemma 5.2]{NPTV}, and a similar argument applies to $S_{V,\eps}$. Therefore, for $\eps<2^{-10}\min\{t_W,t_V\}$, \eqref{S8eq1} turns into
\[ \sup_{\|f_j\|_{L^2(\R^d;\R^n)}=1}
\sum_{Q\in  \mathcal{S}} |Q| \l V^{\frac12}f_1\r_{1+\eps,Q}  \l W^{\frac12} f_2 \r_{1+\eps,Q} \lesssim \sqrt{[W,V]_{A_2} [W]_{A_\infty} [V]_{A_\infty} } \]   which in turn proves Proposition \ref{matrixform}.
\end{proof}

\begin{remark} We may derive a slightly stronger weighted estimate than \eqref{weightedMAIN} for the non-maximally truncated rough integrals $T_{\Omega,\delta}$, by applying Theorem \ref{theoremABSvv} in conjunction with the $(1,1+\eps)$ localized estimates proved in \cite{CoCuDPOu}. Namely, the estimate
\[
\left\|    W^{\frac12} T_{\Omega,\delta} (V^{\frac12}f )  \right\|_{L^2(\R^d;\R^n)}
  \lesssim   \sup_{\|g\|_{L^2(\R^d;\R^n)}=1} \inf_{\eps>0} \sup_{\mathcal S}   \frac{1}{\eps}  \sum_{Q\in  \mathcal{S}} |Q| \l V^{\frac12}f\r_{1,Q}  \l W^{\frac12} g \r_{1+\eps,Q}. 
\]
holds uniformly in $\delta>0$ for all $f$ of unit norm in $L^2(\R^d; \R^n)$. Repeating the  proof of Proposition \ref{matrixform} then yields the slightly improved weighted estimate
\[\sup_{\delta>0}\left\|   W^{\frac12} T_{\Omega,\delta} (V^{\frac12}f ) \right\|_{{L^2(\R^d;\R^n)}} \lesssim \min\{[W]_{A_\infty}, [V]_{A_\infty} \} \sqrt{[W,V]_{A_2} [W]_{A_\infty} [V]_{A_\infty} }\left\|f \right\|_{L^2(\R^d; \R^n)}.
\]
\end{remark}

\subsection{Proof of Theorem \ref{theoremABSvv}} The proof of Theorem \ref{theoremABSvv} is formally identical to the argument for the scalar valued case, provided that    estimate \eqref{stopiter1tilde} and the definition of $E_Q$ given in \eqref{EQ} are replaced by suitable vector valued versions. We begin with the second tool. The proof, which is a minor variation on  \cite[Lemma 3.3]{CuDPOu2}, is given below
\begin{lemma} \label{vvmaxth}
Let $0<\eta\leq 1$,  $Q$ be a dyadic cube and $f_j\in L^{p_j}(\R^d; \R^n)$, $j=1,2$.  Then the set 
 \begin{equation}
\label{EQ2}
E_Q:=\bigcup_{j=1}^2\left\{ x \in 3Q: \eta \l f_j\cic{1}_{3Q}\r_{p_j, L} \not\subset     \l f_j \r_{p_j,3Q} \textrm{ \emph{for some cube} } L\subset \R^d \textrm{ with } x \in L \right\}
\end{equation}
satisfies $|E_Q|\leq C \eta^{\min\{p_1,p_2\}}   |Q|$ for some absolute dimensional constant $C$.
\end{lemma}
\begin{proof} We may assume that $\supp f_j\subset 3Q$. It is certainly enough to estimate the measure of each $j\in \{1,2\}$ component of $E_Q$ by $C\eta^{p_j}|Q|$, and we do so: we fix $j$ and  are thus free to write $f_j=f,p_j=p$. Let $\mathcal L_{f}=\{ L\subset \R^d: \eta \l f\r_{p, L} \not\subset     \l f \r_{p,3Q} \}$.  By usual covering arguments it suffices to show that if $L_1,\ldots L_m\in \mathcal L_{f}$ are disjoint then
\[
\sum_{\mu=1}^m |L_\mu| \leq C \eta^p|Q|.
\] Fix such a disjoint collection $L_1,\ldots L_m$.
Notice that if $A\in \mathrm{GL}(\R^n)$ then $\mathcal L_{Af}=\mathcal L_{f}$.  By action of   $\mathrm{GL}(\R^n)$ we may thus reduce to the case where $\mathcal{E}_{\l f \r_{p,3Q}}=B_1$, and in particular
\begin{equation}
\label{jonell}
B_1 \subset  \l f \r_{p,3Q} \subset B_{ \sqrt{n} }.
\end{equation}
 By membership of each $L_\mu\in \mathcal L_{f}$, we know that $\eta \l f \r_{p,L_\mu}\not\subset B_1 $. A fortiori, there exists $\varphi_\mu\in \Phi_{p'}(L_\mu)$ and a coordinate index $\ell_\mu\in \{1,\ldots,n\}$ such that
\begin{equation}
\label{etadef}
 \eta(F_\mu)_\ell>\frac{1}{\sqrt{n}},\qquad  F_\mu:=\int_{L_\mu} f \varphi_\mu \, \frac{\d x}{|L_\mu|}.
\end{equation}
Let $M_\ell=\{\mu\in \{1,\ldots,m\}: \ell_\mu=\ell \}$. As $\{1,\ldots,m\}=\cup\{M_\ell: \ell=1,\ldots, n\}$ it suffices to show that
\begin{equation}
\label{etafin}
\frac{1}{|3Q|}\sum_{\mu \in M_\ell} |L_\mu|  =:\delta< C\eta^p. 
\end{equation}
Using the membership $\varphi_\mu\in \Phi_{p'}(L_\mu)$ for the first inequality and   the disjointness of the supports for the second equality 
\[
1 \geq  \frac{1}{\delta}\sum_{\mu \in M_\ell } \int_{L_\mu\cap 3Q} |\varphi_{\mu}|^{p'}  \, \frac{\d x}{|3Q|} = \int_{3Q} |\varphi|^{p'}  \, \frac{\d x}{|3Q|}, \qquad \varphi:=\delta^{-\frac{1}{p'}} \sum_{\mu \in M_\ell }   \varphi_{\mu}\cic{1}_{3Q}
\] 
so that $\varphi \in \Phi_{p'}(3Q)$. In particular,  beginning with the right inclusion in \eqref{jonell} and using \eqref{etadef} in the last inequality
\[
\sqrt{n} \geq  \int_{3Q} (f \varphi)_\ell \, \frac{\d x}{|3Q|}= \delta^{-\frac{1}{p'}} \sum_{\mu \in M_\ell} \frac{|L_\mu|}{|3Q|} (F_\mu)_\ell > \frac{1}{\eta\sqrt{n}} \delta^{\frac1p}
\]
which rearranging yields \eqref{etafin} with $C=n^p$, thus completing the proof.
\end{proof}
At this point, let  $ \mathcal S_k$ be a collection of pairwise disjoint cubes as in Step 3 of the proof of Theorem \ref{theoremABS}. The elements of the collection $\mathcal S_{k+1}$ are defined to  be the   maximal dyadic cubes $L$ such that the same condition as in \eqref{EQ} holds, provided the definition of $E_Q$ therein is replaced with the one in \eqref{EQ2}. By virtue of Lemma \ref{vvmaxth}, \eqref{stopiter5} still holds  provided $\eta$ is chosen small enough. And,  we still obtain that $\mathcal S_{k+1}(Q)=\{L \in \mathcal S_{k+1}: L\subset 3Q \}$ is a stopping collection. By definition of $\mathcal S_{k+1}$, it must be that 
\begin{equation}
\label{vvpfeq1}
  \l f_j\cic{1}_{3Q}\r_{p_j, K} \subset  C    \l f_j \r_{p_j,3Q}  
\end{equation} 
whenever the (not necessarily dyadic) cube $K$ is such that a moderate dilate $CK$ of $K$ contains $2^5 L$ for some $  L \in \mathcal S_{k+1}(Q).$
Fix $Q$ for a moment and let $A_j=(A_j^{m\mu}:1\leq m,\mu\leq n)\in \mathrm{GL}(\R^n),j=1,2$ be  chosen such that the John ellipsoid of   $\l \widetilde{f_j}\r_{3Q,p_j}$ is $B_1${, or its intersection with a lower dimensional subspace in a degenerate case,}  and $A_{j} \widetilde{f_j}:= f_j$. It follows from \eqref{vvpfeq1} that if $  2^5L\subset CK $ then $\l \widetilde {f_j}\cic{1}_{3Q}\r_{p_j, K} \subset  B_{C}$. This fact, together with Lemma \ref{simplelemma} readily yields the estimates  \begin{equation}
\label{vvpf2}
\|\widetilde{f_j}\|_{\mathcal Y_{p_j}(\mathcal Q)} \lesssim 1, \qquad j=1,2. 
\end{equation}
We are ready to obtain a substitute for \eqref{stopiter1tilde}.
In fact
\[\left|\left\l
T[K]^{t_2\wedge s_Q}_{t_1}\otimes \mathrm{Id}_{\R^n}(f_1\cic{1}_{Q}),f_2\right\r\right| \leq |Q|\left|\sum_{m=1}^n\mathcal Q[K]^{t_2 }_{t_1}  (f_{1m} ,f_{2m}) \right|  +\sum_{\substack{L \in  \mathcal S_{k+1}  (Q)\\ L \subset Q}}
\left|\left\l
T[K]^{t_2\wedge s_L}_{t_1}\otimes \mathrm{Id}_{\R^n}(f_1\cic{1}_{L}),f_2\right\r\right|
\]
and  by actions of $\mathrm{GL}(\R^n)$, see the proof of \cite[Lemma 3.4]{CuDPOu2}, \[\begin{split} &\quad 
\left|\sum_{m=1}^n\mathcal Q[K]^{t_2 }_{t_1}  (f_{1m} ,f_{2m})\right| = \left|\sum_{m,\mu_1,\mu_2=1}^nA_{1}^{m\mu_1} A_{2}^{m\mu_2}\mathcal Q[K]^{t_2 }_{t_1}  (\widetilde{f_{1\mu_1}} ,\widetilde{f_{2\mu_2}}) \right|
\\&\lesssim \l f_1\r_{p_1, 3Q}\l f_2\r_{p_2, 3Q} \sup_{\mu_1,\mu_2} |\mathcal Q[K]^{t_2 }_{t_1}  (\widetilde{f_{1\mu_1}} ,\widetilde{f_{2\mu_2}})| \lesssim  \l f_1\r_{p_1, 3Q}\l f_2\r_{p_2, 3Q} 
\end{split}
\]
where we also employed \eqref{stopiter0} coupled with \eqref{vvpf2} in the last line. Assembling together the last two displays yields the claimed vector-valued version of \eqref{stopiter1tilde}, and finishes the proof of Theorem \ref{theoremABSvv}.

\section{Proof of Theorem \ref{cor}}\label{sec:quantitative}
We begin with the proof of the first point. As a direct application of the main result of  \cite{Li2},
\[
|\langle {T} f, g\rangle|\le  c_{d,p} \varepsilon^{-1}  [v]_{A_r}^{\frac 1{1+\varepsilon}-\frac 1{p'}}([u]_{A_\infty}^{\frac 1p}+ [v]_{A_\infty}^{\frac 1{p'}})\|f\|_{L^p(w)}\|g\|_{L^{p'}(\sigma)},
\]
where
\[
r= \textstyle \left( \frac{(1+\varepsilon)'}{p}\right)' (\frac p{1+\varepsilon}-1)+1=p+ \frac{p(p-2)\varepsilon }{1-(p-1)\varepsilon},\qquad 
v=\sigma^{\frac{1+\varepsilon}{1+\varepsilon-p'}}=w^{1+\frac {\varepsilon p'}{p'-(1+\varepsilon)}}, \qquad u=w^{\frac {1+\varepsilon}{1+\varepsilon-p}}
\]
By definition,
\begin{align*}
[v]_{A_r}^{\frac 1{1+\varepsilon}-\frac 1{p'}}&= \sup_Q \Big(\frac 1{|Q|}\int_Q w^{1+\frac {\varepsilon p'}{p'-(1+\varepsilon)}}\Big)^{\frac 1{1+\varepsilon}-\frac 1{p'}} \Big(\frac 1{|Q|}\int_Q \sigma^{1+\frac{\varepsilon p}{p-(1+\varepsilon)}}  \Big)^{(r-1)(\frac 1{1+\varepsilon}-\frac 1{p'})}\\
&= \sup_Q \Big(\frac 1{|Q|}\int_Q w^{1+\frac {\varepsilon p'}{p'-(1+\varepsilon)}}\Big)^{\frac 1p \frac 1{1+ \frac {\varepsilon p'}{p'-(1+\varepsilon)}}} \Big(\frac 1{|Q|}\int_Q  \sigma^{1+\frac{\varepsilon p}{p-(1+\varepsilon)}}  \Big)^{\frac 1{p'}\cdot \frac{1}{1+\frac{\varepsilon p}{p-(1+\varepsilon)}}}.
\end{align*}
By the sharp reverse H\"older inequality   \cite{HPR}, taking $\varepsilon= \frac{1}{\tau_d\max\{p, p'\}\max\{[w]_{A_\infty},[\sigma]_{A_\infty}\}}$, we can conclude
\[
\|{T}\|_{L^p(w)}\le c_{d,p} [w]_{A_p}^{\frac 1p}([w]_{A_p}^{\frac 1{p'}}+[\sigma]_{A_\infty}^{\frac 1p}) \max\{[\sigma]_{A_\infty}, [w]_{A_\infty}\}\le c_{d,p} [w]_{A_p}^{2\max\{1,\frac 1{p-1}\}}.
\]

Next let us prove the Fefferman-Stein type inequality of the second point. Indeed,
let $A(t)=t^{pr/{\tilde r}}$ and $\bar B(t)=t^{\frac 12(\frac p{\tilde r}+1)}$, where $1<r<p$ and $\tilde r=\frac{pr-\frac{r-1}2}{pr-(r-1)}$. Then
\begin{align*}
\sup_Q \langle w^r \rangle_Q^{\frac 1{pr}}\|(M_{r}w)^{-\tilde r/p}\|_{B, Q}^{\frac1{\tilde r}}&\le \sup_Q\inf_{x\in Q} (M_{r}w)^{\frac 1p}\|(M_{r}w)^{-\tilde r/p}\|_{B, Q}^{\frac1 {\tilde r}}\le 1.
\end{align*} Let $v=M_{r} w$. Now we have,
\begin{align*}
\tilde r'\sum_{Q\in \mathcal S}|Q|\langle f\rangle_{\tilde r, Q}\langle gw^{\frac 1p} \rangle_{\tilde r, Q}&= \tilde r' \sum_{Q\in \mathcal S}\langle f^{\tilde r} v^{\frac {\tilde r}p} v^{-\frac {\tilde r}p}\rangle_Q^{\frac 1{\tilde r}}\langle gw^{\frac 1p} \rangle_{\tilde r, Q}|Q|\\
&\le \tilde r' \sum_{Q\in \mathcal S}\| f^{\tilde r} v^{\frac {\tilde r}p} \|_{\bar B, Q}^{\frac 1{\tilde r}} \|v^{-\frac {\tilde r}p}\|_{B, Q}^{\frac 1{\tilde r}} \|w^{\frac {\tilde r}p}\|_{A, Q}^{\frac 1{\tilde r}} \|g^{\tilde r}\|_{\bar A, Q}^{\frac 1{\tilde r}}|Q|\\
&\le 2 \tilde r' \sum_{Q\in \mathcal S}\| f^{\tilde r} v^{\frac {\tilde r}p} \|_{\bar B, Q}^{\frac 1{\tilde r}}\|g^{\tilde r}\|_{\bar A, Q}^{\frac 1{\tilde r}} |E_Q|\\
&\le 2\tilde r' \int M^{\mathcal D}_{\bar B}(f^{\tilde r} v^{\frac {\tilde r}p})^{\frac 1{\tilde r}} M^{\mathcal D}_{\bar A(t^{\tilde r})}(g) \\
&\le 2 \tilde r'\| M^{\mathcal D}_{\bar B}(f^{\tilde r} v^{\frac {\tilde r}p})^{\frac 1{\tilde r}}\|_{L^p} \| M^{\mathcal D}_{\bar A(t^{\tilde r})}(g)\|_{L^{p'}}\\
&\le c_d p^2 (p')^{\frac 1{p}} (r')^{1+\frac 1{p'}}\|f\|_{L^p(v)}\|g\|_{L^{p'}}.
\end{align*}
By sparse domination formula and duality,
\[
\|{T}(f)\|_{L^p(w)}\le  c_d p^2 (p')^{\frac 1{p}} (r')^{1+\frac 1{p'}}\|f\|_{L^p(M_r w)}.
\]
Notice that the $A_1$-$A_\infty$ estimate just follows from the sharp reverse H\"older inequality, so that we may restrict to $q>1$. The idea is still viewing $A_q$ condition as a bumped $A_p$ condition (see \cite{Li3}). Let $C(t)=t^{\frac p{r(q-1)}}$. We have
\begin{align*}
r' \sum_{Q\in \mathcal S}|Q|\langle f\rangle_{ r, Q}\langle g w \rangle_{  r, Q}\le r' \sum_{Q\in \mathcal S}|Q|\langle f^r w^{\frac rp}\rangle_{\bar C, Q}^{\frac 1r} \langle w^{-\frac rp} \rangle_{C, Q}^{\frac 1r}  \langle g^{rs} w \rangle_{Q}^{\frac 1{rs}}\langle w^{(r-\frac 1s)s'}\rangle_Q^{\frac 1{rs'}}.
\end{align*}
Take
\[
r=1+\frac{1}{8p(\frac pq)'\tau_d [w]_{A_\infty}}, \,\,s=1+\frac{1}{4(\frac pq)'p}.
\]
Then $rs<1+\frac 1{2p}<p'$, $r< 1+\frac 18(\frac pq-1)<\frac pq$ and $(r-\frac 1s)s'<1+ \frac 1{\tau_d[w]_{A_\infty}}$. Then applying the sparse domination, and the   sharp reverse H\"older inequality we obtain
\begin{align*}
\|{T}(f)\|_{L^p(w)}&\le \sup_{\|g\|_{L^{p'}(w)}=1}r' \sum_{Q\in \mathcal S}|Q|\langle f\rangle_{ r, Q}\langle g w \rangle_{  r, Q}\\
&\le \sup_{\|g\|_{L^{p'}(w)}=1} c_{d,p,q} [w]_{A_\infty} \sum_{Q\in \mathcal S}|Q|\langle f^r w^{\frac rp}\rangle_{\bar C, Q}^{\frac 1r}   \langle g  \rangle_{rs, Q}^w\langle w \rangle_Q  \langle w^{-\frac 1{q-1}} \rangle_Q^{\frac{q-1}p}\\
&\le \sup_{\|g\|_{L^{p'}(w)}=1}c_{d,p,q} [w]_{A_\infty} [w]_{A_q}^{\frac 1p}\sum_{Q\in \mathcal S}\langle f^r w^{\frac rp}\rangle_{\bar C, Q}^{\frac 1r}   \langle g  \rangle_{rs, Q}^w w(Q)^{\frac 1{p'}} |Q|^{\frac 1p}\\
&\le\sup_{\|g\|_{L^{p'}(w)}=1} c_{d,p,q} [w]_{A_\infty} [w]_{A_q}^{\frac 1p}\big(\sum_{Q\in \mathcal S}\langle f^r w^{\frac rp}\rangle_{\bar C, Q}^{\frac 1r}     |Q| \big)^{\frac 1p}\big(\sum_{Q\in \mathcal S}(\langle g  \rangle_{rs, Q}^w)^{p'} w(Q) \big)^{\frac 1{p'}}\\
&\le c_{d,p,q} [w]_{A_\infty}^{1+\frac 1{p'}} [w]_{A_q}^{\frac 1p} \|f\|_{L^p(w)},
\end{align*}where in the last step we have used the Carleson embedding theorem; we omit the routine details.

Finally, we prove the Coifman-Fefferman type inequality. Fix $\varepsilon>0$ and denote $\eta=1+\varepsilon$. Also let
\[r= 1+\frac{1}{8p\eta' \tau_d [w]_{A_\infty}},\,\, s=1+\frac 1{4\eta' p}.
\]Then again $rs<1+\frac 1{2p}<p'$, $r<\eta$ and $(r-\frac 1s)s'<1+ \frac 1{\tau_d[w]_{A_\infty}}$. Applying the sparse domination formula again, we obtain
\begin{align*}
\|{T}(f)\|_{L^p(w)}&\le \sup_{\|g\|_{L^{p'}(w)}=1}r' \sum_{Q\in \mathcal S}|Q|\langle f\rangle_{ r, Q}\langle g w \rangle_{  r, Q}\lesssim \sup_{\|g\|_{L^{p'}(w)}=1}  \eta' [w]_{A_\infty}\sum_{Q\in \mathcal S}  \langle f\rangle_{\eta, Q}\langle g\rangle_{rs, Q}^w w(Q)\\
&\lesssim \sup_{\|g\|_{L^{p'}(w)}=1}  \eta' [w]_{A_\infty}^2 \int_{\mathbb R^d}   M_\eta f M_{rs, w}^{\mathcal D} (g) w dx
\lesssim    \eta' [w]_{A_\infty}^2 \| M_\eta f\|_{L^p(w)}.
\end{align*}

\bibliography{DiPlinioHytonenLi2017}
\bibliographystyle{amsplain}
\end{document}